\documentclass[12pt,a4paper,final]{amsart}
\usepackage{amssymb}
\usepackage{hyperref}
\usepackage[notcite, notref]{showkeys}
\usepackage[utf8]{inputenc}
\usepackage{rotating}

\usepackage{color}

\usepackage{listings} 
\lstdefinelanguage{Sage}[]{Python}
{morekeywords={False,sage,True},sensitive=true}
\definecolor{dblackcolor}{rgb}{0.0,0.0,0.0}
\definecolor{dbluecolor}{rgb}{0.01,0.02,0.7}
\definecolor{dgreencolor}{rgb}{0.2,0.4,0.0}
\definecolor{dgraycolor}{rgb}{0.30,0.3,0.30}

\theoremstyle{plain}
\newtheorem{theorem}{Theorem}[section]
\newtheorem{proposition}[theorem]{Proposition}

\newtheorem{lemma}[theorem]{Lemma}

\theoremstyle{definition}

\newtheorem{remark}[theorem]{Remark}
\newtheorem{question}[theorem]{Question}

\theoremstyle{remark}

\numberwithin{equation}{section}

\newcommand{\N}{\mathbb N}
\newcommand{\Z}{\mathbb Z}

\newcommand{\R}{\mathbb R}
\newcommand{\C}{\mathbb C}

\newcommand{\Hy}{\mathbb H}

\newcommand{\fa}{\mathfrak a}

\newcommand{\fg}{\mathfrak g}
\newcommand{\fh}{\mathfrak h}

\newcommand{\fk}{\mathfrak k}

\newcommand{\fm}{\mathfrak m}

\newcommand{\fp}{\mathfrak p}
\newcommand{\fq}{\mathfrak q}

\newcommand{\ft}{\mathfrak t}

\DeclareMathOperator{\GL}{GL}

\DeclareMathOperator{\SO}{SO}
\DeclareMathOperator{\SU}{SU}
\DeclareMathOperator{\Sp}{Sp}
\DeclareMathOperator{\Ut}{U}

\DeclareMathOperator{\Spin}{Spin}

\newcommand{\gl}{\mathfrak{gl}}
\newcommand{\sll}{\mathfrak{sl}}
\newcommand{\so}{\mathfrak{so}}

\newcommand{\op}{\operatorname}
\newcommand{\Id}{\textup{Id}}

\newcommand{\mi}{\mathrm{i}}

\DeclareMathOperator{\tr}{Tr}
\DeclareMathOperator{\diag}{diag}
\DeclareMathOperator{\Span}{Span}

\newcommand{\innerdots}{\langle {\cdot},{\cdot}\rangle }
\newcommand{\ee}{\varepsilon}

\DeclareMathOperator{\Hom}{Hom}

\DeclareMathOperator{\Ad}{Ad}

\DeclareMathOperator{\Cas}{Cas}

\DeclareMathOperator{\Spec}{Spec}

\DeclareMathOperator{\scal}{scal}

\DeclareMathOperator{\Rc}{Rc}

\DeclareMathOperator{\vol}{vol}
\newcommand{\kil}{\textup{B}}

\newcommand{\PP}{\mathcal{P}}

\newcommand{\Scomplex}[1][(n+1)]{\op{S}_{\C}(\R^{2{#1}})}
\newcommand{\Squaternionic}[1][(n+1)]{\op{S}_{\Hy}(\C^{2{#1}})}
\newcommand{\Grass}[2]{\op{Gr}_{{#1}}(\R^{{#2}})}

\title{Spectrally distinguishing symmetric spaces II}

\author{Emilio~A.~Lauret}
\address{Instituto de Matemática (INMABB), Departamento de Matemática, Universidad Nacional del Sur (UNS)-CONICET, Bahía Blanca, Argentina.}
\email{emilio.lauret@uns.edu.ar}

\author{Juan~Sebastián~Rodríguez}
\address{Departamento de Matemáticas, Pontificia Universidad Javeriana, Bogotá, Colombia.}
\email{js.rodriguez@javeriana.edu.co}

\date{\today}

\begin{document}

\begin{abstract}
The action of the subgroup $\operatorname{G}_2$ of $\operatorname{SO}(7)$ (resp.\ $\operatorname{Spin}(7)$ of $\operatorname{SO}(8)$) on the Grassmannian space $M=\frac{\operatorname{SO}(7)}{\operatorname{SO}(5)\times\operatorname{SO}(2)}$ (resp.\ $M=\frac{\operatorname{SO}(8)}{\operatorname{SO}(5)\times\operatorname{SO}(3)}$) is still transitive. 
We prove that the spectrum (i.e.\ the collection of eigenvalues of its Laplace-Beltrami operator) of a symmetric metric $g_0$ on $M$ coincides with the spectrum of a $\operatorname{G}_2$-invariant (resp.\ $\operatorname{Spin}(7)$-invariant) metric $g$ on $M$ only if $g_0$ and $g$ are isometric.  

As a consequence, each non-flat compact irreducible symmetric space of non-group type is spectrally unique among the family of all currently known homogeneous metrics on its underlying differentiable manifold. 
\end{abstract} 

\maketitle


\section{Introduction}

Any Riemannian manifold $(M,g)$ has naturally associated a distinguished second order differential operator called the Laplace-Beltrami operator $\Delta_g$. 
When $M$ is compact, its spectrum $\Spec(M,g):=\Spec(\Delta_g)$ is real, non-negative and discrete. 
Two compact Riemannian manifolds $(M_1,g_1)$ and $(M_2,g_2)$ are called \emph{isospectral} if $\Spec(M_1,g_1)=\Spec(M_2,g_2)$. 

It is expected that compact Riemannian manifolds with distinguished geometrical properties are spectrally unique, that is, any isospectral Riemannian manifold is necessarily isometric. 
In this article we consider the compact symmetric spaces as geometrically distinguished manifolds. 
Since this problem is still very difficult (e.g.\ it is not known whether a round sphere of dimension at least $7$ is spectrally unique among orientable Riemannian manifolds), we restrict the family to compact homogeneous Riemannian manifolds. 
We refer to the first part of this series, \cite{LauretRodriguez-hearingsymm1}, for a recent account of previous results on this subject. 

This article focuses in the following particular and natural question: 
\begin{question}\label{question}
Is any non-flat compact irreducible symmetric space $(M,g)$ spectrally unique within the space of homogeneous Riemannian metrics on $M$?
\end{question}
The non-flat assumption is due to the existence of isospectral and non-isometric flat tori. 
We next summarize partial answers to Question~\ref{question}.
The cases of compact rank one symmetric space were solved in \cite{BLPhomospheres}
(see \cite{SchmidtSutton13, Lauret-SpecSU(2), LinSchmidtSuttonII} for the particular cases of $S^3$ and $P^3(\R)$).

The cases $M=K\cong\frac{K\times K}{\diag(K)}$, with $K$ a compact simple Lie group, has the great difficulty that in most cases it is not know whether the space of left-invariant metrics $\mathcal M_{\text{left}}(K)$ includes all homogeneous Riemannian metrics on $K$. 
Moreover, the problem is already difficult restricted to the family $\mathcal M_{\text{left}}(K)$. 
Gordon, Schueth and Sutton~\cite{GordonSchuethSutton10} proved that any symmetric (i.e.\ bi-invariant) metric $g_0$ on $K$ is spectrally isolated within $\mathcal M_{\text{left}}(K)$, that is, there is a neighborhood $V$ around $g_0$ such that no metric in $V\smallsetminus\{g_0\}$ is isospectral to $g_0$. 
Furthermore, the only particular cases fully solved are $\SU(2)\simeq S^3$, $\SO(3)\simeq P^3(\R)$, and $\Sp(n)$ for any $n\geq1$ by \cite{Lauret-globalrigid}.  

The remaining cases are compact irreducible symmetric spaces of rank $\geq2$ of non-group type (i.e.\ it has a symmetric presentation $\bar G/\bar K$ satisfying that $\bar G$ is simple). 
Again, the space of homogeneous metrics on the corresponding underlying differentiable manifold is not classified for most of cases. 
The only cases that we know they admit homogeneous metrics are the following:
\begin{equation}\label{eq:espaciossimetricos}
\begin{aligned}
\Scomplex[n] &:= \frac{\SO(2n)}{\Ut(n)}
	\cong \frac{\SO(2n-1)}{\Ut(n-1)},\qquad & 
\Grass{2}{7} &:= \frac{\SO(7)}{\SO(2)\times\SO(5)}
	\cong \frac{\op{G}_2}{\Ut(2)},
\\
\Squaternionic[n] &:= \frac{\SU(2n)}{\Sp(n)}
	\cong \frac{\SU(2n-1)}{\Sp(n-1)},& 
\Grass{3}{8} &:= \frac{\SO(8)}{\SO(3)\times\SO(5)}
	\cong \frac{\Spin(7)}{\SO(4)}. 
\end{aligned}
\end{equation}
The two at the left are the space of orthogonal complex structures on $\R^{2n}$ and the space of quaternionic structures on $\C^{2n}$ compatible with the Hermitian metric, which are of type DIII and AII respectively. 
The two at the right are the Grassmannians of oriented real $2$-dimensional subspaces of $\R^7$ and real $3$-dimensional subspaces of $\R^8$, both of type BDI. 

The first presentation $\bar G/\bar K$ for each case in \eqref{eq:espaciossimetricos} is the symmetric presentation, satisfying at the Lie algebra level that $\bar \fg=\bar\fk\oplus\bar\fm$ and $[\bar\fm,\bar\fm]\subset\fk$. 
The second presentation $G/H$ allows us to describe the currently know homogeneous metrics on each of them, as the space of $G$-invariant metrics on $G/H$. 
In each case, $G$ is a subgroup of $\bar G$ such that its action on $\bar G/\bar K$ is still transitive. 
These metrics were discovered by Onishchik~\cite{Onishchik66-inclusion}.

Question~\ref{question} for the symmetric spaces at the left in \eqref{eq:espaciossimetricos} has been answered affirmatively in \cite{LauretRodriguez-hearingsymm1}. 
The main goal of this article is to complete the cases at the right hand side in \eqref{eq:espaciossimetricos}, to obtain the following main result.

\begin{theorem}\label{thm0:main}
Let $G/H$ denote the second presentation of the symmetric space $M=\bar G/\bar K$ of any of the cases in \eqref{eq:espaciossimetricos}. 
If a $G$-invariant metric on $M$ is isospectral to a symmetric metric on $M$, then they are isometric. 
\end{theorem}

It is highly expect that any homogeneous and non-symmetric metric on the underlying differentiable manifold of a compact irreducible symmetric spaces of non-group type and rank $\geq2$ is isometric to a $G$-invariant metric of $G/H$ as in \eqref{eq:espaciossimetricos} (see e.g.\ \cite[Rem.~1.3]{Kerr96}). 
If this is the case, then Theorem~\ref{thm0:main} combined with \cite{BLPhomospheres} imply that the answer to  Question~\ref{question} is affirmative for all compact irreducible symmetric spaces of non-group type of arbitrary rank.

We next move to an application. 
Cao, Hamilton and Ilmanen (see \cite[Thm.~1.1]{CaoHe15}) proved that a compact Einstein manifold $(M,g)$ is $\nu$-unstable if
\begin{equation}
\lambda_1(M,g)<2E,
\end{equation} 
where $E$ is the Einstein constant of $(M,g)$ (i.e.\ $\Rc(g)=Eg$) and $\lambda_1(M,g)$ denotes the smallest positive eigenvalues of the Laplace-Beltrami operator of $(M,g)$. 
Furthermore, Kröncke~\cite[Thm.~1.3]{Kroencke15} proved that a $\nu$-unstable Einstein manifold of positive scalar curvature is necessarily dynamically unstable. 
See \cite{Kroencke15} for their definitions. 

There are (up to positive scalars) two non-symmetric $\op{G}_2$-invariant Einstein metrics on the space $\Grass{2}{7}$. 
We proved in Subsection~\ref{subsec3:nu-stability} that one of them is $\nu$-unstable. 
Similarly, we prove in Subsection~\ref{subsec4:nu-stability} that the non-symmetric $\Spin(7)$-invariant Einstein metrics on $\Grass{3}{8}$ are $\nu$-unstable.

\subsection*{Acknowledgments}
The authors wish to express their gratitude to the referee for a thorough reading and several detailed corrections that have improved the article. 
Also, they are indebted to Ilka Agricola, Marcos Salvai, and Wolfang Ziller, for providing references and helpful comments concerning the embedded complex simple Lie subalgebra of type $G_2$ in $\so(7,\C)$, and to Fiorela Rossi Bertone for her collaboration in proving several small Lie theoretical statements in Section~\ref{sec:case4}. 

The first-named author was supported by grants from FONCyT (PICT-2019-01054), CONICET (PIP 11220210100343CO) and SGCYT--UNS. 
The second-named author was supported by Pontificia Universidad Javeriana through the research project ID 20500 of the call for postdoctoral positions.

\section{Preliminaries} \label{sec:preliminaries} 

Almost all preliminaries contents necessary for this article are in \cite[\S2]{LauretRodriguez-hearingsymm1}. 
We begin this section by summarizing them.

Let $G$ be a compact Lie group and $H\subset G$ a closed subgroup, with Lie algebras $\fg$ and $\fh$ respectively. 
Let $\fp$ be an $\Ad(H)$-invariant complement of $\fh$ in $\fg$, thus $\fg=\fh\oplus\fp$ and $[\fh,\fp]\subset\fp$.

We fix a $G$-invariant metric $g$ on $G/H$ associated with an $\Ad(H)$-invariant metric $\innerdots$ on $\fp$.  
The spectrum of the Laplace-Beltrami operator $\Delta_g$ associated to $(G/H,g)$ is given by 
\begin{equation}\label{eq:spec}
\Spec(G/H, g) :=\Spec(\Delta_g) 
= \bigcup_{\pi \in\widehat G} 
\Big\{\!\!\Big\{ 
	\underbrace{\lambda_j^{\pi}(g),\dots, \lambda_j^{\pi}(g)}_{d_\pi\text{\rm-times}}: 1\leq j\leq d_\pi^H
\Big\}\!\!\Big\},
\end{equation}
where $\widehat G$ is the set of equivalence classes of irreducible representations of $G$, 
$d_\pi=\dim V_\pi$, $d_\pi^H= \dim V_\pi^H$, and $\lambda_1^{\pi}(g), \dots,\lambda_{d_\pi^H}^{\pi}(g)$ are the eigenvalues of the self-adjoint linear endomorphism
\begin{equation}\label{eq:pi(-C_g)}
\pi(-C_g)|_{V_\pi^H}:= 
-\textstyle{\sum\limits_{i=1}^{\dim\fp} }\pi(X_i)^2 
\Big|_{V_{\pi}^H} 
: V_{\pi}^H\longrightarrow V_{\pi}^H,
\end{equation} 
where $\{X_1,\dots,X_{\dim\fp}\}$ is any orthonormal basis of $\fp$ with respect to $\innerdots$ and $C_g=\sum_{i=1}^n X_i^2$. 
The operator $\pi(-C_g)|_{V_\pi^H}$ is uniquely determined by $g$, though the element $C_g$, which lies in the universal enveloping algebra $\mathcal U(\fg)$, is not well defined; see \cite[Rem.~2.2]{LauretRodriguez-hearingsymm1}.

Let $\kil$ be an $\Ad(G)$-invariant inner product on $\fg$. 
We next introduce strong hypotheses that will hold in the next two sections. 
We assume that there exist a closed subgroup $K$ of $G$, with Lie algebra $\fk$, such that $H\subset K\subset G$.
Moreover, we also assume that there are $H$-invariant subspaces $\fp_1,\fp_2,\fp_3$ of $\fp$ such that $\fp=\fp_1\oplus\fp_2\oplus\fp_3$, 
\begin{align}\label{eq:hipotesis}
\fk&=\fh\oplus\fp_3, &
\fg&=\fk\oplus (\fp_1\oplus\fp_2),&
&\text{ with $\fp_1\oplus\fp_2$ invariant by $K$. }
\end{align}

For any $r=(r_1,r_2,r_3)\in\R_{>0}^3$, we set
\begin{equation}
\innerdots_r 
= \frac{1}{r_1^2} \kil|_{\fp_1} \oplus 
  \frac{1}{r_2^2} \kil|_{\fp_2} \oplus 
  \frac{1}{r_3^2} \kil|_{\fp_3}
.
\end{equation}
It follows that the inner product $\innerdots_r$ is $\Ad(H)$-invariant, so it induces a $G$-invariant metric $g_r$ on $G/H$.

We fix $r\in\R_{>0}^3$, $\pi\in\widehat G$ and $v\in V_\pi^H$. 
For each $h=1,2,3$, let $\{X_1^{(h)},\dots,X_{p_h}^{(h)}\}$ be an orthonormal basis of $\fp_h$ with respect to $\kil|_{\fp_h}$ ($p_h=\dim \fp_h$).
It follows that $\bigcup_{h=1}^3 \{r_h X_1^{(h)},\dots,r_h X_{p_h}^{(h)}\}$ is an orthonormal basis of $\fp$ with respect $\innerdots_r$ and therefore 
\begin{equation}\label{eq:pi(C_r)}
\begin{aligned}
\pi(-C_{g_r})\cdot v &
=-\sum_{i_1=1}^{p_1} r_1^2\,  \pi(X_{i_1}^{(1)})^2\cdot v
 -\sum_{i_2=1}^{p_2} r_2^2\,  \pi(X_{i_2}^{(2)})^2\cdot v
 -\sum_{i_3=1}^{p_3} r_3^2\,  \pi(X_{i_3}^{(3)})^2\cdot v
\\ & 
= r_1^2\, \left(
	-\sum_{i_1=1}^{p_1} \pi(X_{i_1}^{(1)})^2 
	-\sum_{i_2=1}^{p_2} \pi(X_{i_2}^{(2)})^2 
	-\sum_{i_3=1}^{p_3} \pi(X_{i_3}^{(3)})^2\right) \cdot v
\\ &\quad 
- (r_2^2-r_1^2) \sum_{i_2=1}^{p_2} \pi(X_{i_2}^{(2)})^2 \cdot v
+ (r_3^2-r_1^2) \left( -\sum_{i_3=1}^{p_3} \pi(X_{i_3}^{(3)})^2 -\sum_{j=1}^{\dim\fh} \pi(Z_j^2)\right) \cdot v
\\ & 
= r_1^2\, \pi\big(-\Cas_{\fg,\kil}\big) \cdot v
+ (r_2^2-r_1^2) \, \Upsilon_\pi(v) 
+ (r_3^2-r_1^2)\, \pi\big(-\Cas_{\fk,\kil|_{\fk}}\big) \cdot v
, 
\end{aligned}
\end{equation}
where $\{Z_1,\dots,Z_{\dim\fk}\}$ is any orthonormal basis of $\fh$ with respect to $\kil|_\fh$, 
\begin{equation}\label{eq:trickyterm}
\Upsilon_\pi(v)= -\sum_{i_2=1}^{\dim\fp_2} \pi(X_{i_2}^{(2)})^2 \cdot v
\end{equation}
(the authors have kindly called it the `tricky term'), and $\Cas_{\fg,\kil}$ (resp.\ $\Cas_{\fk,\kil|_{\fk}}$) is the Casimir operator of $\fg$ (resp.\ $\fk$) with respect to $\kil$ (resp.\ $\kil|_\fk$). 
In the second identity we used that $\pi(Z)\cdot v=0$ for all $Z\in\fh$ because $v\in V_\pi^H$.

It is well known that Casimir elements acts on irreducible representations by calculable scalars; see \cite[\S2.2]{LauretRodriguez-hearingsymm1} for a rigorous definition and its properties.
In particular $\Cas_{\fg,\kil}\cdot v=\lambda_{\kil}^\pi\, v$ for all $v\in V_\pi$, with  $\lambda_{\kil}^\pi =\kil^*(\Lambda_\pi,\Lambda_\pi+2\rho_\fg)$, where $\Lambda_\pi$ is the highest weight of $\pi$ (once a maximal torus of $G$ and a Weyl chamber are chosen). 
However, it is quite difficult to obtain an eigenbasis of $\pi(-C_{g_r})$ in this generality since it may occur that $\pi(-\Cas_{\fk,\kil|_{\fk}})$ and $\Upsilon_\pi$ do not necessarily diagonalize simultaneously. 
Consequently, we do not expect an explicit description of $\Spec(M,g_r)$. 
However, the next remark determines the eigenvalue contributed by $\pi\in\widehat G$ such that $\dim V_\pi^H=1$, which will be enough for our purpose.

\begin{remark}\label{rem:Upsilon-dimV_pi^H=1}
Let $\pi$ be an irreducible representation of $G$ such that $\dim V_\pi^H=1$. 
On the one hand, the condition $\dim V_\pi^H=1$ forces there is exactly one representation $\tau\in\widehat K$ occurring in $\pi|_{K}$ (i.e.\ $\Hom_K(V_\tau,V_\pi)\neq0$) satisfying that $V_\tau^H\neq0$. 
Thus $\pi(-\Cas_{\fk,\kil|_{\fk}})\cdot v= \lambda_{\kil|_\fk}^\tau$ for all $v\in V_\pi^H$. 

On the other hand, $\Upsilon_\pi$ preserves $V_\pi^H$ since $\pi(-C_{g_r})$, $\pi(-\Cas_{\fg,\kil})$, and $\pi(-\Cas_{\fk,\kil|_{\fk}})$ do it. 
Hence, $\Upsilon_\pi$ acts on $V_\pi^H$ by an scalar, say $\upsilon^\pi$. 

We conclude from \eqref{eq:spec} and \eqref{eq:pi(C_r)} that $\pi$ contributes to $\Spec(G/H,g_r)$ with the eigenvalue 
\begin{equation}\label{eq:threeterms}
\begin{aligned}
\lambda_1^{\pi}(r) &
= \lambda_{\kil}^\pi \, r_1^2
+ \upsilon^\pi\,  (r_2^2-r_1^2) 
+ \lambda_{\kil|_\fk}^\tau \, (r_3^2-r_1^2)
\\ & 
= \big(\lambda_{\kil}^\pi -\upsilon^\pi -\lambda_{\kil|_{\fk}}^\tau \big)\, r_1^2
+ \upsilon^\pi \, r_2^2
+ \lambda_{\kil|_\fk}^\tau \, r_3^2
,
\end{aligned}
\end{equation}
with multiplicity $\dim V_\pi$. 
\end{remark}

\section{The case \texorpdfstring{$\Grass{2}{7}$}{Gr(2,7)}} 
\label{sec:case3}

In this section we consider the compact irreducible symmetric space $\Grass{2}{7}$ of oriented two-planes in $\R^7$.

\subsection{Homogeneous metrics for $\Grass{2}{7}$}
Let
\begin{equation}
\Grass{2}{7} =\frac{\SO(7)}{\SO(5)\times\SO(2)}
.
\end{equation} 
Since this presentation is symmetric, the isotropy representation is irreducible and consequently every $\SO(7)$-invariant metric on $\Grass{2}{7}$ is symmetric.  
We next define a non-symmetric presentation $\Grass{2}{7}=G/H$ having a three-parameter family of $G$-invariant metrics, which are of course homogeneous.

Let $G$ be the (unique up to conjugation) subgroup of $\SO(7)$ with Lie algebra of exceptional type ${G}_2$. 
It is well known that the action of $G$ on $\Grass{2}{7}$ is still transitive and the isotropy subgroup $H$ at the trivial element is isomorphic to $\Ut(2)$.

Let $T$ be a maximal subgroup of $H$, which is also a maximal torus of $G$ since $\op{rank}(G)=\op{rank}(H)=2$. 
As usual, we denote by $\ee_1,\ee_2,\ee_3$ the elements satisfying that $\Pi(\fg_\C,\ft_\C) = \{\alpha_1:=\ee_2-\ee_3,\; \alpha_2:=\ee_1-2\ee_2+\ee_3\}$ is a simple root system and $\ft_\C^*=\{\sum_{i=1}^3a_i\ee_i: a_1,a_2,a_3\in\C,\, a_1+a_2+a_3=0\}$. 
This gives fundamental weights $\omega_1:=\ee_1-\ee_3$, $\omega_2:=2\ee_1-\ee_2-\ee_3$, and the positive root system
\begin{equation} \label{eq3:positiverootsystemG}
\begin{aligned}
\Phi^+(\fg_\C,\ft_\C) &= \left\{
	\begin{array}{ll}
	\ee_2-\ee_3,&\ee_1-2\ee_2+\ee_3,\\
	\ee_1-\ee_2,&\ee_1+\ee_2-2\ee_3,\\
	\ee_1-\ee_3,&2\ee_1-\ee_2-\ee_3
	\end{array}
\right\}.
\end{aligned}
\end{equation}
It will be useful the root space decomposition
\begin{equation} \label{eq3:rootspacedecomposition}
\fg_\C =\ft_\C\oplus \bigoplus_{\alpha\in \Phi^+(\fg_\C,\ft_\C)} (\fg_\alpha\oplus\fg_{-\alpha})
.
\end{equation}

Without loosing generality, we pick $H$ the subgroup of $G$ such that its Lie algebra $\fh$ satisfies  $\fh=\ft_\C\oplus\fg_\beta\oplus\fg_{-\beta}$, where $\beta=\ee_1+\ee_2-2\ee_3\in \Phi^+(\fg_\C,\ft_\C)$. 

Let $K$ be the connected subgroup of $G$ such that its Lie algebra $\fk$ satisfies 
\begin{equation}\label{eq3:positiverootsystemK1}
\Phi^+(\fk_\C,\ft_\C) = \left\{
	\ee_1-2\ee_2+\ee_3,\;
	\ee_1+\ee_2-2\ee_3,\;
	2\ee_1-\ee_2-\ee_3
\right\}
,
\end{equation} 
which is isomorphic to $\SU(3)$. 
One can see that the corresponding simple roots are $\beta_1:=\ee_1+\ee_2-2\ee_3$ and $\beta_2:=\ee_1-2\ee_2+\ee_3$, and the fundamental weights are $\nu_1:=\ee_1-\ee_3$, $\nu_2:=\ee_1-\ee_2$.

We pick $\kil=-\kil_\fg$ as our $\Ad(G)$-invariant inner product, where $\kil_\fg$ is the Killing form of $\fg$, 

Let $\fq$ denote the orthogonal complement of $\fk$ into $\fg$. 
It turns out that $\fq$ is irreducible as a $K$-module, or in other words, $G/K$ is an isotropy irreducible space (see for instance \cite[7.107]{Besse}). 
From \eqref{eq3:rootspacedecomposition} and \eqref{eq3:positiverootsystemK1}, it follows that 
$\fq_\C=
\fg_{\ee_1-\ee_2}\oplus \fg_{-\ee_1+\ee_2}
\oplus 
\fg_{\ee_1-\ee_3} \oplus \fg_{-\ee_1+\ee_3} \oplus 
\fg_{\ee_2-\ee_3} \oplus \fg_{-\ee_2+\ee_3}
$.
However, as an $H$-module, we have the decomposition $\fq=\fp_1\oplus\fp_2$ with $\fp_1,\fp_2$ irreducible, and
\begin{align}
(\fp_1)_\C&
=\fg_{\ee_1-\ee_3}\oplus 
 \fg_{-\ee_1+\ee_3} \oplus 
 \fg_{\ee_2-\ee_3} \oplus 
 \fg_{-\ee_2+\ee_3} ,
 &
(\fp_2)_\C&
=\fg_{\ee_1-\ee_2} \oplus \fg_{-\ee_1+\ee_2}. 
\end{align}
Let $\fp_3$ be the orthogonal complement of $\fh$ in $\fk$, which is irreducible as an $H$-module and 
\begin{equation}
(\fp_3)_\C 
=\fg_{\ee_1-2\ee_2+\ee_3} \oplus
 \fg_{-\ee_1+2\ee_2-\ee_3} \oplus \fg_{2\ee_1-\ee_2-\ee_3}\oplus \fg_{-2\ee_1+\ee_2+\ee_3}.
\end{equation}
Note $\dim\fp_1=4$, $\dim\fp_2=2$, $\dim\fp_3=4$.

The decomposition $\fp=\fp_1\oplus\fp_2\oplus\fp_3$ satisfies condition \eqref{eq:hipotesis} in Section~\ref{sec:preliminaries}. 
Moreover, the subspaces $\fp_1,\fp_2,\fp_3$ are irreducible and non-equivalent as $H$-modules, thus every $G$-invariant metric on $\Grass{2}{7}=G/H$ is isometric to $g_r$ for some $r=(r_1,r_2,r_3)\in\R_{>0}^3$, which is induced by the  $\Ad(H)$-invariant inner product on $\fp$ given by 
\begin{equation}
\innerdots_r 
= \frac{1}{r_1^2} \kil|_{\fp_1} 
\oplus \frac{1}{r_2^2} \kil|_{\fp_2} 
\oplus \frac{1}{r_3^2} \kil|_{\fp_3}
.
\end{equation}

\subsection{The tricky term for $\Grass{2}{7}$}
This subsection is devoted to express the tricky term $\Upsilon_{\pi}$ given in \eqref{eq:trickyterm}.

For $\xi\in \ft_\C^*$, let us denote by $u_\xi\in\ft_\C$ the only element in $\ft_\C$ such that $\xi(H)=\kil_\fg(H,u_\xi)$ for all $H\in\ft_\C$. 
Theorem 6.6 in \cite{Knapp-book-beyond} ensures that we can pick $X_\alpha\in\fg_\alpha$ for each $\alpha\in \Phi(\fg_\C,\ft_\C)$ such that
\begin{equation}\label{eq:corchetesKnapp}
[X_\alpha,X_\beta] = 
\begin{cases}
u_\alpha &\text{if }\alpha+\beta=0,\, \alpha>0\\
N_{\alpha,\beta} X_{\alpha+\beta} &\text{if }\alpha+\beta\in \Phi(\fg_\C,\ft_\C),\\
0 &\text{otherwise, }
\end{cases}
\end{equation}
for all $\alpha,\beta\in \Phi(\fg_\C,\ft_\C)$, with constant terms $N_{\alpha,\beta}$ satisfying $N_{\alpha,\beta}=-N_{-\alpha,-\beta}$ and  $N_{\alpha,\beta}^2=\tfrac12 q(1+p)|\alpha|^2$, where $\{\beta+n\alpha: -p\leq n\leq q\}$ is the $\alpha$-string containing $\beta$, and moreover, 
\begin{equation}
\fg=
\bigoplus_{\alpha\in\Pi(\fg_\C,\ft_\C)} \R\mi u_\alpha
\oplus
\bigoplus_{\alpha\in\Phi^+(\fg_\C,\ft_\C)} \big(\R  (X_\alpha-X_{-\alpha}) \oplus \R \mi (X_\alpha+X_{-\alpha})\big)
.
\end{equation}

One can easily check that the following elements form an orthonormal basis of $\fp_2$:
\begin{align}
X_1^{(2)} &=\tfrac1{\sqrt2} (X_{\ee_1-\ee_2}-X_{-\ee_1+\ee_2}),
&
X_2^{(2)} &=\tfrac{\mi}{\sqrt2} (X_{\ee_1-\ee_2}+X_{-\ee_1+\ee_2})
.
\end{align}

\begin{lemma}\label{lem3:Upsilon-rootexpression}
For any $\pi\in \widehat G$, we have that 
$
\Upsilon_\pi(v)
= 2\, \pi(X_{\ee_1-\ee_2})\cdot \big(\pi(X_{-\ee_1+\ee_2})\cdot v\big)
$ 
for any $v\in V_\pi^H$. 
\end{lemma}

\begin{proof}
An easy computation shows that 
\begin{equation*}
\begin{aligned}
\Upsilon_\pi(v)&
=
\pi(X_{-\ee_1+\ee_2})\cdot \pi(X_{\ee_1-\ee_2}) \cdot v  
+ \pi(X_{\ee_1-\ee_2})\cdot \pi(X_{-\ee_1+\ee_2}) \cdot v  
\\ & 
= 2\, \pi(X_{\ee_1-\ee_2})\cdot \big(\pi(X_{-\ee_1+\ee_2})\cdot v\big)
-\pi(u_{\ee_1-\ee_2})\cdot v
.
\end{aligned}
\end{equation*}

Now, $T\subset H$ forces $V_\pi^H\subset V_\pi^T$, which implies that $\pi(u_{\ee_1-\ee_2})\cdot v=0$, and the assertion follows.  
\end{proof}

We now assume $\dim V_\pi^H=1$. 
We set 
\begin{equation}
\begin{aligned}
\fa &=\Span_\C\{u_{\ee_1-\ee_2}, X_{\ee_1-\ee_2}, X_{-\ee_1+\ee_2}\},\\
W&=\Span_\C\{\pi(X_{\ee_1-\ee_2})^l\cdot v,  \pi(X_{-\ee_1+\ee_2})^l\cdot v: l\geq0,v\in V_\pi^H \}
.
\end{aligned}
\end{equation}
It turns out that $\fa$ is a Lie algebra isomorphic to $\sll(2,\C)$ and $W$ is an irreducible $\fa$-module; the last part is not true if $\dim V_\pi^H>1$. 
Moreover, $\dim W$ is odd because its zero weight is non-zero. 
The next goal is to obtain the scalar for which $\Upsilon_\pi$ acts on $V_\pi^H$ in terms of $\dim W$. 

We write $\sll(2,\C)=\Span_\C\{h,e,f\}$ with 
$h=\left(\begin{smallmatrix}
1&0\\ 0&-1
\end{smallmatrix}\right),
$
$e=\left(\begin{smallmatrix}
0&1\\ 0&0
\end{smallmatrix}\right),
$
$
f=\left(\begin{smallmatrix}
0&0\\ 1&0
\end{smallmatrix}\right)
$.
For $m\in\N$, let $(W_m,\chi_m)$ denote the irreducible representation of $\sll(2,\C)$ of dimension $m+1$, which is unique up to equivalence. 
Its weight decomposition is given by $W_m=\bigoplus_{i=0}^m W_m(m-2i)$, with $W_m(m-2i)$ the weight space of weight $m-2i$, which has dimension one, for any $i=0,\dots,m$.

\begin{lemma}\label{lem3:irrep-sl(2,C)}
For $0\leq i\leq m$ and $v\in W_m(m-2i)$, one has
$
2\,\chi_m(e)\cdot(\chi_m(f)\cdot v)
=2(i+1)(m-i)\, v.
$
\end{lemma}

\begin{proof}
From \cite[Thm.~1.63]{Knapp-book-beyond}, there is a basis $\{v_0,\dots,v_m\}$ of $W_m$ such that 
$\chi_m(h)\cdot v_i=(m-2i)v_i$, 
$\chi_m(e)\cdot v_0=0$, 
$\chi_m(f)\cdot v_i=v_{i+1}$ (with $v_{m+1}=0$), and
$\chi_m(e)\cdot v_i=i(m-i+1)v_{i-1}$.
Note that $W_m(m-2j)=\C v_j$ for any $j$. 
We pick $v=v_i$ without loosing generality. 
Then, $2\,\chi_m(e)\cdot(\chi_m(f)\cdot v)=2\chi_m(e)\cdot v_{i+1} = 2(i+1)(m-i)v_{i}$, as required. 
\end{proof}

One can check that the correspondence
\begin{align}
h_{\ee_1-\ee_2}&\leftrightarrow 24\, u_{\ee_1-\ee_2},&
e_{\ee_1-\ee_2}&\leftrightarrow \sqrt{24} \, X_{\ee_1-\ee_2}, &
f_{\ee_1-\ee_2}&\leftrightarrow \sqrt{24} \, X_{-\ee_1+\ee_2}
\end{align}
defines an isomorphism between $\fa$ and $\sll(2,\C)$. 
Since $W(0)=V_\pi^H\neq0$, $\dim W$ is odd and Lemma~\ref{lem3:irrep-sl(2,C)} gives
$
2\,\pi(e)\cdot(\pi(f)\cdot v) 
=\frac{(\dim W)^2-1}{2}\, v
$ 
for any $v\in V_\pi^H$.
We conclude from Lemma~\ref{lem3:Upsilon-rootexpression} that
\begin{equation}\label{eq3:Upsilon_pi(v)}
\Upsilon_\pi(v)
=2\, \pi(X_{\ee_1-\ee_2}) \cdot \big( \pi(X_{-\ee_1+\ee_2})\cdot v\big)
= \frac{(\dim W)^2-1}{48}\, v
\qquad\text{for any }v\in V_\pi^H,
\end{equation}
provided $\dim V_\pi^H=1$.

\subsection{Some low Laplace eigenvalues of $\Grass{2}{7}$}

According to the description \eqref{eq:spec} of the spectrum of the Laplace-Beltrami operator on $(\Grass{2}{7},g_r)$, each $\pi\in\widehat G_H$ contributes to $\Spec(\Grass{2}{7},g_r)$ with $\dim V_\pi^H\dim V_\pi$ eigenvalues. 
The goal of this subsection is to determine these eigenvalues for $\pi_{\omega_1}$ and $\pi_{\omega_2}$. 
We first determine the Casimir eigenvalues. 

\begin{remark}
By combining the branching laws from $G=\operatorname{G}_2$ to $K=\SU(3)$ by Mashimo~\cite{Mashimo97b} and the spherical representations of the pair $(K,H)= (\SU(3),\Ut(2))$, one can prove that $\widehat G_H=\widehat G$, that is, every irreducible representation of $G$ has non-trivial elements invariant by $H$. 
The proof is not included for shortness, since it is not need it for our purposes. 
\end{remark}

\begin{lemma}\label{lem3:casimir}
For $\Lambda=a_1\ee_1+a_2\ee_2+a_3\ee_3\in\PP^+(G)$ and $\mu=b_1\ee_1+b_2\ee_2+b_3\ee_3\in\PP^+(K)$, we have that 
\begin{equation}
\begin{aligned}
\lambda_{\kil_\fg}^{\pi_\Lambda}&=  \lambda^{\pi_\Lambda}
	&&\qquad\text{where }&\lambda^{\pi_\Lambda}&= \frac{1}{24}\big(a_1^2+a_2^2+a_3^2 + 6a_1-2a_2-4a_3\big)
,\\
\lambda_{\kil_\fg|_{\fk}}^{\tau_{\mu}} &
=\frac{3}{4} \lambda^{\tau_{\mu}}
	&&\qquad\text{where } &\lambda^{\tau_{\mu}}&=\frac{1}{18}\big(b_1^2+b_2^2+b_3^2 + 6b_1\big)
.
\end{aligned}
\end{equation}
\end{lemma}

\begin{proof}
One has $\kil_\fg^*(\ee_i,\ee_j)=\frac{1}{24} \delta_{i,j}$ and $\rho_\fg=3\ee_1-\ee_2-2\ee_3$, thus $\lambda^{\pi_{\Lambda}} 
= \kil_\fg^*(\Lambda,\Lambda+2\rho_\fg)
= \tfrac{1}{24}\big(a_1(a_1+6)+a_2(a_2-2)+a_3(a_3-4)\big)
= \frac{1}{24}\big(a_1^2+a_2^2+a_3^2 + 6a_1-2a_2-4a_3\big)
$, as claimed.

By \cite[p.~37]{DAtriZiller}, 
$\kil_\fk
= \frac34\kil_\fg|_{\fk}$, thus $\lambda_{\kil_\fg|_{\fk}}^{\tau_{\mu}} 
=\frac{3}{4} \lambda^{\tau_{\mu}}$ by (2.8) in \cite[\S2.2]{LauretRodriguez-hearingsymm1} and $\kil_{\fk}^*(\ee_i,\ee_j)=\frac{4}{3} \kil_{\fg}^*(\ee_i,\ee_j)
=\frac{1}{18}\delta_{i,j}$ (the factor $\frac34$ is inverted after dualizing; see \cite[(2.7)]{LauretRodriguez-hearingsymm1}). 
Since
$
\rho_\fk
=\frac12 \sum_{\beta\in\Phi^+(\fk_\C,\ft_\C)} \beta 
=2\ee_1-\ee_2-\ee_3
$, we obtain that
$
\lambda^{\tau_{\mu}} 
= -\kil_\fk^*(\mu,\mu+2\rho_\fk)
= -\kil_\fk^*(\mu,\mu+2\rho_\fk)
= \frac1{18}\big(b_1(b_1+4)+b_2(b_2-2)+b_3(b_3-2)\big)
= \frac1{18}\big(b_1^2+b_2^2+b_3^2 + 6b_1\big)
$.
Note that this is consistent with $\lambda^{\Ad_K}=1$ since $\Ad_K=\tau_{\nu_1+\nu_2}$ (i.e.\ the adjoint representation of $K$ has highest weight $\nu_1+\nu_2=2\ee_1-\ee_2-\ee_3$). 
\end{proof}

\begin{proposition}\label{prop3:loweigenvalues}
The representations $\pi_{\omega_1}$ and $\pi_{\omega_2}$ contribute to $\Spec(\Grass{2}{7},g_r)$ with the eigenvalues
\begin{align}
\lambda^{\pi_{\omega_1}}(r)&
	=\frac13r_1^2+\frac16r_2^2&\text{ and }&&
\lambda^{\pi_{\omega_2}}(r)&
	=\frac1{12}r_1^2+\frac16r_2^2+\frac{3}{4}r_3^2,
\end{align}
with multiplicity $7$ and $14$ respectively.  
\end{proposition}

\begin{proof}
We have the weight decomposition
\begin{equation}
\begin{aligned}
V_{\pi_{\omega_1}}&= V_{\pi_{\omega_1}}(0) \oplus \bigoplus_{i\neq j} V_{\pi_{\omega_1}}(\ee_i-\ee_j)
.
\end{aligned}
\end{equation}
One can easily check that $V_\pi^H= V_\pi^K=V_{\pi_{\omega_1}}(0)$, so $\tau=1_K$ is the only irreducible representation of $K$ satisfying $[1_H:\tau|_H][\tau:\pi_{\omega_1}|_K]>0$. 
Moreover, seeing $V_{\pi_{\omega_1}}$ as an $\fa$-module, the irreducible subspace containing $V_\pi^H$ is  $W:=V_{\pi_{\omega_1}}(\ee_1-\ee_2)\oplus V_{\pi_{\omega_1}}(0)\oplus V_{\pi_{\omega_1}}(-\ee_1+\ee_2)$. 
Since $\dim W=3$, \eqref{eq3:Upsilon_pi(v)} forces $\Upsilon_\pi|_{V_{\pi}^H}=\frac{1}{6}\Id_{V_{\pi}^H}$, or $\upsilon_{\kil}^\pi=\frac16$ in the notation of Remark~\ref{rem:Upsilon-dimV_pi^H=1}. 
Furthermore, $\lambda^{\pi_{\omega_1}}=\frac12$ by Lemma~\ref{lem3:casimir}. 
According to Remark~\ref{rem:Upsilon-dimV_pi^H=1}, we conclude that $\pi_{\omega_1}$ contributes with the eigenvalue
\begin{equation}
\begin{aligned}
\lambda^{\pi_{\omega_1}}(r) &
= \big(\lambda_{\kil}^{\pi_{\omega_1}} -\upsilon_{\kil}^{\pi_{\omega_1}} -\lambda_{\kil|_\fk}^{1_K} \big)r_1^2
+ \upsilon_{\kil}^{\pi_{\omega_1}} r_2^2
+ \lambda_{\kil|_\fk}^{1_K} r_3^2
= \tfrac13 r_1^2
+ \tfrac16 r_2^2
,
\end{aligned}
\end{equation}
with multiplicity $\dim V_{\pi_{\omega_1}}=7$, as claimed. 

We now consider $\pi_{\omega_2}$, which is equivalent to the adjoint representation of $G$, thus $\lambda^{\pi_{\omega_2}}=1$ and its non-zero weights are precisely the roots in $\Phi(\fg_\C,\ft_\C)$. 

One can get from the branching rule from $G$ to $K$ in \cite{Mashimo97b} that 
\begin{equation}\label{eq3:pi_omega2|_K}
\pi_{\omega_2}|_{K}\simeq \tau_{\nu_1}\oplus\tau_{\nu_2}\oplus\tau_{\nu_1+\nu_2}.
\end{equation}
Alternatively, \cite{Sage} calculates it as follows: 
\smallskip
\begin{lstlisting}
sage: G=WeylCharacterRing("G2", style="coroots")
sage: K=WeylCharacterRing("A2", style="coroots")
sage: b=branching_rule(G,K,"extended")
sage: omega=G.fundamental_weights()
sage: print("checking the dimension of G(omega2):", G(omega[2]).degree())
sage: print("branching G(omega2) to K:")
sage: G(omega[2]).branch(K,rule=b)
checking the dimension of G(omega2): 14
branching G(omega2) to K:
A2(0,1) + A2(1,0) + A2(1,1)
\end{lstlisting}
\smallskip
Here, for non-negative integers $a,b$, \texttt{A2(a,b)} means in our notation $\tau_{a\nu_1+b\nu_2}$.

The first two terms in \eqref{eq3:pi_omega2|_K} are the standard and its contragradient representation, and none of them contains non-trivial fixed points by $H$. 
The representation $\tau_{\nu_1+\nu_2}$ is precisely the adjoint representation of $K$, and its $H$-invariant subspace are the elements in the Cartan subalgebra $\ft$ orthogonal to $\fh$, that is, $V_{\pi_{\omega_2}}^H=V_{\tau_{\nu_1+\nu_2}}^H=\R\mi u_{\ee_1-\ee_2}$. 
Lemma~\ref{lem3:casimir} gives $\lambda_{\kil|_\fk}^{\tau_{\nu_1+\nu_2}}=\frac34$. 

The $\fa$-irreducible subspace of $V_{\pi_{\omega_2}}\simeq\fg$ containing $V_\pi^H$ is precisely $\fa$, which has dimension $3$, so $\Upsilon_{\pi_{\omega_2}}|_{V_\pi^H}=\frac16 \Id_{V_\pi^H}$. 
We conclude that $\pi_{\omega_2}$ contributes with the eigenvalue
\begin{equation}
\begin{aligned}
\lambda^{\pi_{\omega_2}}(r) &
= \big(\lambda_{\kil}^{\pi_{\omega_2}} -\upsilon_{\kil}^{\pi_{\omega_2}} -\lambda_{\kil|_\fk}^{\tau_{\nu_1+\nu_2}} \big)r_1^2
+ \upsilon_{\kil}^{\pi_{\omega_2}} r_2^2
+ \lambda_{\kil|_\fk}^{\tau_{\nu_1+\nu_2}} r_3^2
= \tfrac1{12} r_1^2
+ \tfrac16 r_2^2
+ \tfrac34 r_3^2
,
\end{aligned}
\end{equation}
with multiplicity $\dim V_{\pi_{\omega_1}}=14$, as claimed. 
\end{proof}

\subsection{Spectral uniqueness for $\Grass{2}{7}$}

We are now ready to show that every symmetric metric on $\Grass{2}{7}$ is spectrally unique withing the space of $G$-invariant metrics on $\Grass{2}{7}$. 

According to \cite[\S5]{Kerr96}, the symmetric metrics on $\Grass{2}{7}$ are 
\begin{equation}
\left\{\bar g_t:= g_{(\sqrt6 t,\sqrt3 t,\sqrt2 t)}:t>0 \right\}.
\end{equation}

\begin{remark}\label{rem3:notationMegan}
In the notation in \cite{Kerr96}, $\fp_1$ and $\fp_2$ are interchanged and $Q=8\, \kil_\fg$, so $x_1=\frac{8}{r_2^2}$, $x_2=\frac{8}{r_1^2}$, and $x_3=\frac{8}{r_3^2}$.
\end{remark}

The standard symmetric space $\big(\frac{\SO(7)}{\SO(5)\times\SO(2)}, g_{\kil_{\so(7)}}\big)$ is isometric to $\bar g_t$ for $t=\sqrt{2/5}$.
One has $\lambda_1\big(\frac{\SO(7)}{\SO(5)\times\SO(2)}, g_{\kil_{\so(7)}}\big)=1$ with multiplicity $21$ since it is attained at the adjoint representation of $\SO(7)$ (see \cite[Table A.2]{Urakawa86}).
Indeed, we note that $\lambda^{\pi_{\omega_1}}(\sqrt{12/5}, \sqrt{6/5}, \sqrt{4/5})=\lambda^{\pi_{\omega_2}}(\sqrt{12/5}, \sqrt{6/5}, \sqrt{4/5})=1$ and $\dim \pi_{\omega_1}+\dim\pi_{\omega_2}=7+14=21$, which implies that any eigenvalue of the Laplace-Beltrami operator of $(\Grass{2}{7},\bar g_{t})$ coming from $\pi\in\widehat G\smallsetminus\{1_G,\pi_{\omega_1},\pi_{\omega_2}\}$ is strictly greater than $\lambda_1(\Grass{2}{7},\bar g_{t})$, for any $t>0$.

\begin{theorem}
Any $G$-invariant metric on $\Grass{2}{7}$ isospectral to a symmetric metric on $\Grass{2}{7}$ is in fact isometric to such symmetric metric. 
\end{theorem}

\begin{proof}
Suppose that $\Spec(\Grass{2}{7},g_{r})=\Spec(\Grass{2}{7},\bar g_t)$ for some $r=(r_1,r_2,r_3)\in\R_{>0}^3$ and $t>0$. 
Without loosing generality, we can assume that $t=1$, that is, $\bar g_t=g_{r_0}$ with $r_0=(\sqrt{6},\sqrt{3},\sqrt{2})$. 
The goal is to show that $r=r_0$.

We mentioned above that the multiplicity of 
$
\bar \lambda_1:=
\lambda_1(\Grass{2}{7},\bar g_{1})
$ in $\Spec(\Grass{2}{7},\bar g_{1})$ is $21$. 
This implies that $\bar \lambda_1$ is in $\Spec(\Grass{2}{7},g_{r})$ with multiplicity $21$, thus \eqref{eq:spec} forces 
\begin{equation}\label{eq3:multiplicidad}
21=\sum_{\pi\in\widehat G_H:\,  \bar\lambda_1\in\Spec(\pi(-C_{g_r}))|_{V_\pi^H} }  \dim V_\pi\, a_\pi,
\end{equation}
where $a_\pi$ denotes the multiplicity of $\bar\lambda_1$ in $\Spec(\pi(-C_{g_r}))|_{V_\pi^H}$, so $0\leq a_\pi\leq \dim V_\pi^H$. 
One can easily check that the only irreducible representations of $G$ of dimension at most $21$ are $1_G$, $\pi_{\omega_1}$, and  $\pi_{\omega_2}$. 
We know that the eigenvalue associated to the trivial representation $1_G$ is $0$. 
Since $\dim V_{\pi_{\omega_1}}^H=\dim V_{\pi_{\omega_2}}^H=1$ (see the proof of Proposition~\ref{prop3:loweigenvalues}), $\dim V_{\pi_{\omega_1}}=7$, and $\dim V_{\pi_{\omega_2}}=14$, we conclude that $a_{\pi_{\omega_1}}=a_{\pi_{\omega_2}}=1$ and $a_{\pi}=0$ for all $\pi\in\widehat G\smallsetminus\{1_G,\pi_{\omega_1},\pi_{\omega_2}\}$. 

Hence 
$
\lambda^{\pi}(r)=\lambda^{\pi}(\sqrt6,\sqrt3,\sqrt2)
$ for $\pi=\pi_{\omega_1},\pi_{\omega_2}$, which is equivalent to 
\begin{equation}\label{eq3:igualdad-autovalores}
\begin{aligned}
\begin{cases}
\tfrac52 = \tfrac13 r_1^2 + \tfrac{1}{6}r_2^2 ,
\\
\tfrac52 =
	\tfrac{1}{12}r_1^2
	+ \tfrac{1}{6} r_2^2
	+ \tfrac{3}{4} r_3^2 ,
\end{cases}
\iff \qquad &
\begin{cases}
\tfrac52 = \tfrac13 r_1^2 + \tfrac{1}{6}r_2^2 ,
\\
r_1^2=3r_3^2.
\end{cases}
\end{aligned}
\end{equation}

We now analyze the volume, which is also determined by the spectrum. 
We have that 
\begin{equation}
\vol(\Grass{2}{7},g_r)
=r_1^{\dim\fp_1} r_2^{\dim\fp_2} r_3^{\dim\fp_3} \vol(\Grass{2}{7},g_{(1,1,1)})
=r_1^{4} r_2^{2} r_3^{4} \vol(\Grass{2}{7},g_{(1,1,1)}).
\end{equation}
Now, $\vol(\Grass{2}{7},g_r)=\vol(\Grass{2}{7},g_{r_0})$ yields $r_1^{4} r_2^{2} r_3^{4} 
=2^{4}3^3$. 
Substituting $r_1^2=3r_3^2$ in it, we obtain that
\begin{equation}\label{eq3:igualdadpto}
r_2^{2} r_3^{8} =2^{4}3 .
\end{equation}
By replacing $r_2^2=15-6r_3^2$ from the first row in \eqref{eq3:igualdad-autovalores} above, we get $0= r_3^{10}-\tfrac{5}{2} r_3^{8} +8$.

The polynomial $f(x):=x^5-\tfrac{5}{2}x^4+8$  satisfies $f(0)=8>0$, $f(2)=0$, $f'(x)=5x^3(x-2)$ and $\lim_{x\to\infty}f(x)=+\infty$.
It follows that $f(x)$ vanishes only at $x=2$ for $x>0$.
We conclude that $r_3=\sqrt{2}$, which implies $r=r_0$ by \eqref{eq3:igualdad-autovalores} and \eqref{eq3:igualdadpto}. 
\end{proof}

\subsection{Homogeneous Einstein metrics on $\Grass{2}{7}$}\label{subsec3:nu-stability}

The Einstein symmetric metric $g_{\kil_{\so(7)}}$ on $\frac{\SO(7)}{\SO(5)\times\SO(2)}$ is neutrally $\nu$-stable according to Cao and He~\cite{CaoHe15}. 
Actually, we mentioned before that $\lambda_1(\frac{\SO(7)}{\SO(5)\times\SO(2)}, g_{\kil_{\so(7)}})=1$ and its Einstein constant is $E=1/2$ by \cite[Prop.~7.93]{Besse}, thus $\lambda_1(\frac{\SO(7)}{\SO(5)\times\SO(2)}, g_{\kil_{\so(7)}})=2E$.

There are two additional (up to scaling) $G$-invariant Einstein metrics on $\Grass{2}{7}$ discovered in \cite{Arvanitoyeorgos93} and \cite{Kimura}; see also \cite[\S5]{Kerr96}.
These two metrics do not belong to a canonical variation of any of the two fibrations, which makes very difficult the calculations of their first eigenvalue.

Taking into account the dictionary of notations in Remark~\ref{rem4:notationMegan} between this article and \cite{Kerr96}, the additional two $G$-invariant Einstein metrics $g_1,g_2$ on $\Grass{2}{7}$ have approximate parameters $x=(x_1,x_2,x_3)$ and $r=(r_1,r_2,r_3)$ given by  
\begin{equation}
\renewcommand{\arraystretch}{1.2}
\begin{array}{ccc}
& g_1 & g_2
\\ \hline 
x_1&1&1
\\
x_2&0.597133339764792&5.35063404291744
\\
x_3&1.22554394913282&5.25152734929540
\\
r_1^2&13.3973427160359&1.49514990855887
\\
r_2^2&8& 8
\\
r_3^2&6.52771367820850&1.52336634047490
\end{array}
\end{equation}

By using the expression for the scalar curvature in \cite[page 164]{Kerr96}, we obtain that 
\begin{equation}
\begin{aligned}
\scal(\Grass{2}{7},g_1)&\approxeq 33.6319213085489,&
\scal(\Grass{2}{7},g_2)&\approxeq 7.25191745508143.
\end{aligned}
\end{equation}
Dividing by $\dim\Grass{2}{7}=10$, we obtain the Einstein constants $E_1\approxeq 3.36319213085489$ and $E_2\approxeq 0.725191745508143$ of $g_1$ and $g_2$ respectively.  

Although we do not have an explicit expression for the first Laplace eigenvalue of $g_1$ and $g_2$, we have that 
\begin{equation}
\begin{aligned}
\lambda_1(\Grass{2}{7},g_1)&\leq \lambda^{\pi_{\omega_1}}(r(g_1))\approxeq 5.79911423867862< 6.72638426170978  \approxeq 2E_1,
\end{aligned}
\end{equation}
thus the Einstein manifold $(\Grass{2}{7},g_1)$ is $\nu$-unstable. 

It is not possible to obtain the same consequence for the Einstein metric $g_2$ because 
\begin{equation}
\begin{aligned}
\lambda^{\pi_{\omega_1}}(r(g_2))&\approxeq 1.83171663618629 \quad \text{and}
&
\lambda^{\pi_{\omega_2}}(r(g_2))&\approxeq 2.60045391440275
\end{aligned}
\end{equation}
are both greater than $2E_2=1.45038349101629$.

\section{The case \texorpdfstring{$\Grass{3}{8}$}{Gr(3,8)}} 
\label{sec:case4}

In this section we consider the compact irreducible symmetric space $\Grass{3}{8}$ of oriented three-planes in $\R^8$.

\subsection{Root systems for $\Grass{3}{8}$}
We consider the compact Lie groups
\begin{equation*}
H:=\SO(4)\subset K:=\operatorname{G}_2\subset G=\Spin(7). 
\end{equation*}
The goal of this section is to describe these embeddings at the Lie algebra level $\fh\subset \fk\subset \fg$ and their complexifications $\fh_\C\subset \fk_\C \subset \fg_\C$. 
We will identify $\fg\equiv \so(7)=\{X\in\gl(7,\R): X^t+X=0\}$

We consider the maximal torus $T$ of $G$ such that its Lie algebra is given by
\begin{equation*}
\ft=\left\{ \diag\left( 
	\left(\begin{smallmatrix}
	0&\mi h_1\\ -\mi h_1&0
	\end{smallmatrix}\right)
	,
	\left(\begin{smallmatrix}
	0&\mi h_2\\ -\mi h_2&0
	\end{smallmatrix}\right)
	,
	\left(\begin{smallmatrix}
	0&\mi h_3\\ -\mi h_3&0
	\end{smallmatrix}\right)
	,0
\right) : h_1,h_2,h_3\in\mi \R
\right\}. 
\end{equation*}
Its complexification $\ft_\C$ has the same expression with $h_1,h_2,h_3\in\C$. 
For $i=1,2,3$, we define $H_i\in\ft_\C$ as above setting $h_i=1$ and $h_j=0$ for $j\neq i$. 
It follows that $\{H_1,H_2,H_3\}$ is a $\C$-basis of $\ft_\C$.
Let $\{\ee_1,\ee_2,\ee_3\}$ be its dual basis, that is, $\ee_i(H_j)=\delta_{i,j}$. 

One can check that $\kil_\fg(X,Y)=-5\tr(XY)$ for all $X,Y\in\fg$, where $\kil_\fg$ stands for the negative of the Killing form of $\fg$. 
We pick $\kil=\kil_\fg$ as the $\Ad(G)$-invariant inner product on $\fg$ fixed in Section~\ref{sec:preliminaries}. 

One has $\kil(H_i,H_j)=10\, \delta_{i,j}$. 
For $\nu=\sum_{i=1}^3 a_i\ee_i\in\ft_\C^*$, the element $u_\nu:=\frac1{10}\sum_{i=1}^{3} a_iH_i$ satisfies $\nu(H)=\kil(H,u_\nu)$ for all $H\in\ft_\C$.  
We extend $\kil|_{\ft_\C}$ to $\ft_\C^*$ by $\kil^*(\nu,\nu')=\kil(u_{\nu},u_{\nu'})$. 

We pick the standard Weyl chamber such that the positive root system is given by 
\begin{equation*}
\Phi^+(\fg_\C,\ft_\C)
=\left\{ \ee_i\pm\ee_j:1\leq i<j\leq 3 \right\}\cup\{\ee_1,\ee_2,\ee_3\}. 
\end{equation*}
The corresponding fundamental weights are $\omega_1=\ee_1$, $\omega_2=\ee_1+\ee_2$, and $\omega_3=\frac12(\ee_1+\ee_2+\ee_2)$. 
For $\alpha\in\Phi(\fg_\C,\ft_\C)$, set $H_{\alpha}=\frac{2u_\alpha}{\kil^*(\alpha,\alpha)}$. 
We have $H_{\pm\ee_i\pm \ee_j}=\pm H_i\pm H_j$ and $H_{\pm\ee_i}=\pm 2 H_i$. 

For $\alpha\in\Phi(\fg_\C,\ft_\C)$, we denote by $E_\alpha$ the element in $\fg_\C$ defined as in Example 2 in \cite[\S{}II.1]{Knapp-book-beyond}. 
We set $X_\alpha=c_\alpha E_\alpha$, where $c_{\ee_i\pm\ee_j}=\frac12$ and $c_{-\ee_i\pm\ee_j}=\frac12$ for $1\leq i<j\leq 3$, and $c_{\pm \ee_i}=\pm 1$ for $i=1,2,3$.  
That is, 
\begin{equation}
\begin{aligned}
X_{\ee_1-\ee_2}&=
\frac12 
\begin{pmatrix}
&&1&\mi\\
&&-\mi&1\\
-1&\mi\\
-\mi&-1\\
&&&&0_3\\
\end{pmatrix},
&
X_{-\ee_1+\ee_2}&=
\frac{-1}2 
\begin{pmatrix}
&&1&-\mi\\
&&\mi&1\\
-1&-\mi\\
\mi&-1\\
&&&&0_3\\
\end{pmatrix},
\\
X_{\ee_1+\ee_2}&=
\frac12 
\begin{pmatrix}
&&1&-\mi\\
&&-\mi&-1\\
-1&\mi\\
\mi&1\\
&&&&0_3\\
\end{pmatrix},
&
X_{-\ee_1-\ee_2}&=
\frac{-1}2 
\begin{pmatrix}
&&1&\mi\\
&&\mi&-1\\
-1&-\mi\\
-\mi&1\\
&&&&0_3\\
\end{pmatrix},
\\
X_{\ee_1} &=
\begin{pmatrix}
0&&&&1\\
&0&&&-\mi\\
&&0_3\\
&&&0\\
-1&\mi&&&0\\
\end{pmatrix}
,
&
X_{-\ee_1} &=-
\begin{pmatrix}
0&&&&1\\
&0&&&\mi\\
&&0_3\\
&&&0\\
-1&-\mi&&&0\\
\end{pmatrix},
\end{aligned}
\end{equation}
and the rest can be figured out by changing indexes. 
Here $0_3$ abbreviates the zero $3\times 3$ matrix. 

This particular choice makes a Chevalley basis, that is, $[X_\alpha, X_{-\alpha}]=H_{\alpha}$ for all $\alpha\in \Phi(\fg_\C,\ft_\C)$,  and 
$[X_{\alpha}, X_{\beta}]=\pm(m+1)\, X_{\alpha+\beta}$ with $m=\max\{a\in\Z: \beta-a\alpha\in \Phi(\fg_\C,\ft_\C)\}$, for all $\alpha,\beta\in \Phi(\fg_\C,\ft_\C)$ satisfying that $\alpha+\beta\neq0$.  

We are now in position to describe the embeddings $\fh_\C\subset \fk_\C\subset \fg_\C$. 
According to \cite[\S2.3]{LevasseurSmith}, we have 
\begin{equation}\label{eq4:k_c}
\fk_\C=\Span_\C\left\{
\begin{aligned}
\bar H_{\alpha_1}&:=H_1-H_2+2H_3, &
Y_{\pm\alpha_1}, &&
Y_{\pm(\alpha_1+\alpha_2)}, &&
Y_{\pm(3\alpha_1+\alpha_2)}, 
\\
\bar H_{\alpha_2}&:=H_2-H_3, &
Y_{\pm\alpha_2}, &&
Y_{\pm(2\alpha_1+\alpha_2)}, &&
Y_{\pm(3\alpha_1+2\alpha_2)}
\end{aligned}
\right\},
\end{equation}
where
\begin{align*}
Y_{\pm\alpha_1}&= X_{\pm(\ee_1-\ee_2)}+X_{\pm\ee_3},&
Y_{\pm\alpha_2}&= X_{\pm(\ee_2-\ee_3)},
\\
Y_{\pm(\alpha_1+\alpha_2)}&=-X_{\pm(\ee_1-\ee_3)}+X_{\pm\ee_2},&
Y_{\pm(2\alpha_1+\alpha_2)}&=-X_{\pm(\ee_2+\ee_3)}-X_{\pm\ee_1}, 
\\
Y_{\pm(3\alpha_1+\alpha_2)}&=-X_{\pm(\ee_1+\ee_3)}, &
Y_{\pm(3\alpha_1+2\alpha_2)}&=-X_{\pm(\ee_1+\ee_2)}.
\end{align*}
Here, $\alpha_1$ is the short simple root in  $\Phi(\fk_\C,(\ft\cap\fk)_\C)$ and $\alpha_2$ is the long one. 

Without loosing generality, we can assume that 
\begin{equation}\label{eq4:h_C}
\begin{aligned}
\fh_\C &
= \Span_\C \{\bar H_{\alpha_1}, Y_{\alpha_1}, Y_{-\alpha_1}\}
\stackrel{\perp}{\oplus} 
\Span_\C \{ 
\bar H_{3\alpha_1+2\alpha_2}:=
-H_1-H_2
, Y_{3\alpha_1+2\alpha_2},  Y_{-3\alpha_1-2\alpha_2} 
\}
.
\end{aligned}
\end{equation}

\subsection{Homogeneous metrics for $\Grass{3}{8}$}
\label{subsec4:nu-stability}
We first find the irreducible components of the isotropy representation of $G/H$. 

Since $\kil(\fg_\alpha,\fg_\beta)=0$ if $\alpha+\beta\neq0$, it follows immediately that the orthogonal complement $\fp_3$ of $\fh$ in $\fk$ with respect to $\kil$ satisfies
\begin{equation}
(\fp_3)_\C
=\Span_\C\left\{ 
Y_{\pm\alpha_2},
Y_{\pm(\alpha_1+\alpha_2)},
Y_{\pm(2\alpha_1+\alpha_2)},
Y_{\pm(3\alpha_1+\alpha_2)}
\right\}.
\end{equation}
Moreover, $\fp_3$ is irreducible as an $H$-module since $(\fp_3)_\C$ is equivalent to $\sigma_3\otimes\sigma_1$, where $\sigma_k$ denotes the irreducible representation of $\sll(2,\C)$ of dimension $k+1$. 

\begin{lemma}
The orthogonal complement of $\fk$ in $\fg$ (with respect to $\kil$) decomposes as irreducible $H$-modules as $\fp_1\oplus\fp_2$, where
\begin{equation}\label{eq:q=p1+p2}
\begin{aligned}
(\fp_1)_\C&
=\Span_\C\left\{ 
X_{\pm(\ee_2+\ee_3)}-\tfrac12 X_{\pm\ee_1},
X_{\pm(\ee_1-\ee_3)}+\tfrac12 X_{\pm\ee_2}
\right\},
\\
(\fp_2)_\C&
=\Span_\C\left\{ 
	H_1-H_2-H_3,\, 
	X_{\pm(\ee_1-\ee_2)}-\tfrac12 X_{\pm\ee_3}
\right\}.
\end{aligned}
\end{equation}
\end{lemma}

\begin{proof}
By using that $[X_\alpha,X_{-\alpha}]= \kil_{\fg}(X_\alpha,X_{-\alpha})\, u_\alpha$ for all $\alpha\in\Phi(\fg_\C,\ft_\C)$ (see for instance \cite[Lem.~2.18(a)]{Knapp-book-beyond}), one can easily check that $\kil_\fg(X_{\ee_i\pm\ee_j}, X_{-(\ee_i\pm\ee_j)})=10$ for all $1\leq i<j\leq 3$, and $\kil_\fg(X_{\ee_i}, X_{-\ee_i})=20$ for all $1\leq i\leq 3$. 
This allows us to prove that $(\fp_1)_\C\oplus(\fp_2)_\C$ is orthogonal to $\fk$ by checking that every generator of $(\fp_1)_\C\oplus(\fp_2)_\C$ is orthogonal to every generator of $\fk$. 
For instance, 
\begin{equation}
\kil_\fg( X_{\ee_2+\ee_3}-\tfrac12 X_{\ee_1}, 
	Y_{-(2\alpha_1+\alpha_2)}
)
= -\kil_\fg( X_{\ee_2+\ee_3}, X_{-\ee_2-\ee_3} )
+ \tfrac12 \kil_\fg( X_{\ee_1}, X_{-\ee_1})
=0
. 
\end{equation}
The rest are very simple. 

We next obtain the decomposition of $(\fp_1)_\C\oplus(\fp_2)_\C$ as irreducible $H$-modules. 
Since $Y_{\pm(3\alpha_1+2\alpha_2)}$ acts trivially on $\fp_2$ and 
\begin{equation}
\begin{aligned} {}
[ Y_{\alpha_1}, X_{-\ee_1+\ee_2}-\tfrac{1}{2} X_{-\ee_3} ]&
=H_{\ee_1-\ee_2}+H_{\ee_3}
=H_1-H_2+2H_3,\\ {}
[Y_{\alpha_1}, H_1-H_2+2H_3]&
= -2(X_{\ee_1-\ee_2}-\tfrac12 X_{-\ee_3}) ,\\ {}
[Y_{\alpha_1},X_{\ee_1-\ee_2}-\tfrac12 X_{\ee_3}]&
= 0,
\end{aligned}
\end{equation}
it follows that $(\fp_2)_\C$ is an irreducible $\fh_\C$-submodule equivalent to $\sigma_2\otimes\sigma_0$. 
Similarly, one can see that $(\fp_1)_\C$ is an irreducible $\fh_\C$-submodule equivalent to $\sigma_1\otimes\sigma_1$. 
\end{proof}

Note that $\dim\fp_1=4$, $\dim\fp_2=3$, $\dim\fp_3=8$, so $\fp_1,\fp_2,\fp_3$ are pairwise non-equivalent as $H$-modules. 
Furthermore, the decomposition $\fp=\fp_1\oplus\fp_2\oplus\fp_3$ satisfies condition \eqref{eq:hipotesis} in Section~\ref{sec:preliminaries}. 
We conclude that every $G$-invariant metric on $\Grass{3}{8}=G/H$ is isometric to $g_r$ for some $r=(r_1,r_2,r_3)\in\R_{>0}^3$, which is induced by the  $\Ad(H)$-invariant inner product on $\fp$ given by 
\begin{equation}
\innerdots_r 
= \frac{1}{r_1^2} \kil|_{\fp_1} 
\oplus \frac{1}{r_2^2} \kil|_{\fp_2} 
\oplus \frac{1}{r_3^2} \kil|_{\fp_3}
.
\end{equation}

\subsection{Casimir scalars and the tricky term for $\Grass{3}{8}$}
The goal of this subsection is to provide explicit expressions for the three individual terms of $\pi(-C_{g_r})$ in \eqref{eq:pi(C_r)}, namely the Casimir operators $\pi(-\Cas_{\fg,\kil})$ and $\pi(-\Cas_{\fk,\kil|_{\fk}})$, and the tricky term $\Upsilon_\pi$. 

We have that the sum of positive roots is $2\rho_\fg=5\ee_1+3\ee_2+\ee_3$. 
If $\pi_{\Lambda}\in\widehat G$ has highest weight $\Lambda=\sum_{i=1}^3a_i\ee_i$, then $\pi_{\Lambda} (-\Cas_{\fg,\kil})$ acts on $V_{\pi_{\Lambda}}$ by the scalar 
\begin{equation}\label{eq4:lambda^pi_Lambda}
\lambda_{\kil_\fg}^{\pi_{\Lambda}}
= \kil_{\fg}^*(\Lambda_\pi,\Lambda_\pi+2\rho_\fg)
= \frac{1}{10}\big(
	a_1(a_1+5)
	+a_2(a_2+3)
	+a_3(a_3+1)\big)
. 
\end{equation}

We have that $\kil_{\fk}=\frac{4}{5}\kil_{\fg}|_\fk$ by \cite[page 37]{DAtriZiller}, thus 
$
\pi(-\Cas_{\fk,\kil_\fg|_{\fk}}) 
= \frac45 \, \pi(-\Cas_{\fk,\kil_\fk}) 
$
(see \cite[\S2.2]{LauretRodriguez-hearingsymm1}). 
By writing the simple roots in $\Phi(\fk_\C,\ft_\C)$ in the usual way $\alpha_1=\bar\ee_2-\bar \ee_3$ and $\alpha_2 =\bar\ee_1-2\bar\ee_2+\bar \ee_3$, one has the fundamental weights $\nu_1=\bar \ee_1-\bar\ee_3$,  $\nu_2=2\bar\ee_1-\bar\ee_2-\bar\ee_3$, and $2\rho_\fk=6\bar\ee_1-2\bar\ee_2-4\bar\ee_3$.
Since $\kil_\fk^*(\bar\ee_i,\bar\ee_j) =\frac{1}{24}\delta_{i,j}$, if $\tau_\nu\in\widehat K$ has highest weight $\nu=\sum_{i=1}^3b_i\bar\ee_i$, then  $\tau_{\nu}(-\Cas_{\fk,\kil})$ acts on $V_{\tau_{\nu}}$ by the scalar 
\begin{equation}\label{eq4:lambda^tau_mu}
\begin{aligned}
\lambda_{\kil}^{\tau_\nu}&= 
\frac{4}{5}\, \kil_{\fk}^*(\nu,\nu +2\rho_\fk)
=\frac{4}{5}\,\frac{1}{24}\Big(
	b_1(b_1+6)
	+b_2(b_2-2)+b_3(b_3-4)
\Big)
. 
\end{aligned}
\end{equation}

We next move to the tricky term. 
One can check that an orthonormal basis of $\fp_2$ is given by the three elements 
\begin{equation}
\begin{aligned}
X_1^{(2)}&:= \tfrac{1}{\sqrt{30}} \big((X_{\ee_1-\ee_2}-\tfrac12 X_{\ee_3}) - (X_{-(\ee_1-\ee_2)} -\tfrac12 X_{-\ee_3}) \big)
,
\\
X_2^{(2)}&:= \tfrac{\mi}{\sqrt{30}}  \big((X_{\ee_1-\ee_2}-\tfrac12 X_{\ee_3}) + (X_{-(\ee_1-\ee_2)} -\tfrac12 X_{-\ee_3}) \big)
,
\\
X_3^{(2)}&:=\tfrac{\mi}{\sqrt{30}}\big(H_1-H_2-H_3\big)
.
\end{aligned}
\end{equation}
Hence, for $\pi\in\widehat G$ and a weight vector $v$ of weight $\mu$ (i.e.\ $v\in V_\pi(\mu)$), we have that 
\begin{equation}\label{eq4:trickyterm}
\begin{aligned}
\Upsilon_\pi(v)&=
-\big(\pi(X_1^{(2)})^2 +\pi(X_2^{(2)})^2 +\pi(X_3^{(2)})^2 \big)\cdot v
\\ & 
= \frac{\mu(H_1-H_2-H_3)^2}{30}\, v
-\frac{1}{15} \left(\begin{array}{l}
	- 2\pi(X_{\ee_1-\ee_2}) \pi(X_{-(\ee_1-\ee_2)}) 
	+ \pi(H_{\ee_1-\ee_2})
	\\
	+ \pi(X_{\ee_1-\ee_2}) \pi(X_{-\ee_3})
	+ \pi(X_{\ee_3}) \pi(X_{-(\ee_1-\ee_2)}) 
	\\
	-\tfrac12 \pi(X_{\ee_3}) \pi(X_{-\ee_3}) 
	+\tfrac14 \pi(H_{\ee_3})
	\\
\end{array}\right)\cdot v
\\ & 
= \frac{1}{30}\Big(
\mu(H_1-H_2-H_3)^2
-\mu(2H_1-2H_2+H_3)
\Big)\, v
\\ & \quad 
+\frac{2}{15} \pi(X_{\ee_1-\ee_2}) \pi(X_{-(\ee_1-\ee_2)})\cdot v
+\frac{1}{30} \pi(X_{\ee_3}) \pi(X_{-\ee_3}) \cdot v
\\ & \quad 
-\frac{1}{15} \pi(X_{\ee_1-\ee_2}) \pi(X_{-\ee_3})\cdot v
-\frac{1}{15} \pi(X_{\ee_3}) \pi(X_{-(\ee_1-\ee_2)}) \cdot v
.
\end{aligned}
\end{equation}

\subsection{Presentations of the standard and the spin representation}\label{subsec4:presentation}

The strategy to calculate the tricky term $\Upsilon_\pi(v)$ for $v\in V_\pi^H$ is to work with a particular presentation of the representation $\pi:G=\Spin(7) \to \GL(V_\pi)$. 
For our purposes, it will be enough to consider the standard representation $\pi_{\omega_1}$ and the spin representation $\pi_{\omega_3}$, which have highest weights $\omega_1=\ee_1$ and $\omega_3=\tfrac12(\ee_1+\ee_2+\ee_3)$ respectively. 

We have $V_{\pi_{\omega_1}}=\C^7$, and the action is multiplication at the left. 
Its weights are $\PP(\pi_{\omega_1})=\{0,\pm\ee_1,\pm\ee_2,\pm\ee_3\}$. 
Let $\{e_1,\dots,e_7\}$ denote the canonical basis of $\C^7$. 
We set $u_{\ee_i}= e_{2i-1}-\mi e_{2i}$ and $u_{-\ee_i}= e_{2i-1}+\mi e_{2i}$ for any $i=1,2,3$, and $u_{0}= e_{7}$. 
It follows that $u_{\mu}$ is a weight vector of weight $\mu$ and $\{u_{\mu}:\mu\in\PP(\pi_{\omega_1})\}$ is a basis of $V_{\pi_{\omega_1}}$. 
Moreover, one can easily check that the non-trivial actions between an element $X_{\alpha}$ with $\alpha\in\Phi(\fg_\C,\ft_\C)$ and $u_{\mu}$ with $\mu\in\PP(\pi_{\omega_1})$ are in Table~\ref{table:standard}. 

\begin{table}
\caption{Representation table of the standard representation $\pi_{\omega_1}$ of $\fg_\C=\so(7,\C)$.} 
\label{table:standard}
$
\begin{aligned}
X_{\ee_i-\ee_j}\cdot u_{\ee_j}&=  u_{\ee_i},&\qquad
X_{\ee_i+\ee_j}\cdot u_{-\ee_j}&=  u_{\ee_i},&\qquad
X_{\ee_k}\cdot u_{0}&=  u_{\ee_k},
\\
X_{\ee_i-\ee_j}\cdot u_{-\ee_i}&=  -u_{-\ee_j},&
X_{\ee_i+\ee_j}\cdot u_{-\ee_i}&=  -u_{\ee_j},&
X_{\ee_k}\cdot u_{-\ee_k}&= -2\, u_0,
\\
X_{-\ee_i+\ee_j}\cdot u_{-\ee_j}&=  -u_{-\ee_i},&
X_{-\ee_i-\ee_j}\cdot u_{\ee_j}&=  -u_{-\ee_i},&
X_{-\ee_k}\cdot u_{0}&=  -u_{-\ee_k},
\\
X_{-\ee_i+\ee_j}\cdot u_{\ee_i}&=  u_{\ee_j},&
X_{-\ee_i-\ee_j}\cdot u_{\ee_i}&=  u_{-\ee_j},&
X_{-\ee_k}\cdot u_{\ee_k}&= 2\, u_0.
\\
\end{aligned}
$

\

for every $1\leq i<j\leq 3$ and $1\leq k\leq 3$.

\end{table}

We now move to the spin representation $\pi_{\omega_3}$. 
There are several presentations of it (see e.g.\ \cite[Ch.~V, Problems 16--27]{Knapp-book-beyond}).
For shortness reasons, we will give a particular basis of weight vectors with the corresponding actions of the basis of $\fg_\C$. 
We have 
\begin{equation}
\PP(\pi_{\omega_3})= \{\tfrac12(\pm\ee_1\pm\ee_2\pm\ee_3)\} = \{\pm\omega_3\}\cup\{\pm(\omega_3-\ee_i):1\leq i\leq 3\}
,
\end{equation}
each of them with multiplicity one. 
One can check that there is a basis $\{v_{\mu}:\mu\in\PP(\pi_{\omega_3})\}$ of $V_{\pi_{\omega_3}}$, with $v_{\mu}\in V_{\pi_{\omega_3}}(\mu)$ for all $\mu$, such that the non-trivial elements of the form $X_{\alpha}\cdot v_{\mu}$ for $\alpha\in\Phi(\fg_\C,\ft_\C)$ are shown in Table~\ref{table:spin}. 
Note that $X_{\alpha}\cdot v_{\mu}\neq0$ if and only if $\alpha+\mu\in\PP(\pi_{\omega_3})$. 
Of course, $\pi_{\omega_3}(H)\cdot v_{\mu}= \mu(H)\, v_{\mu}$ for all $H\in\ft_\C$. 

\begin{remark}
In the notation of \cite[Ch.~V, Problems 19--27]{Knapp-book-beyond}, one has 
$v_{\omega_3}= z_{\emptyset}$, 
$v_{\omega_3-\ee_2}= z_{\{2\}}'$, 
$v_{\ee_3-\omega_3}= z_{\{1,2\}}$, 
$v_{\ee_1-\omega_3}= z_{\{2,3\}}$, 
$v_{\omega_3-\ee_1}= z_{\{1\}}'$, 
$v_{\omega_3-\ee_3}= z_{\{3\}}'$, 
$v_{\ee_2-\omega_3}= z_{\{1,3\}}$, 
$v_{-\omega_3}= z_{\{1,2,3\}}'$. 
\end{remark}

\begin{table}
\caption{Representation table of the spin representation $\pi_{\omega_3}$ of $\fg_\C=\so(7,\C)$.} 
\label{table:spin}
$
\begin{aligned}
X_{\ee_1}\cdot v_{\omega_3-\ee_1}&= -2\, v_{\omega_3}, &
X_{\ee_2}\cdot v_{\omega_3-\ee_2}&= -2\, v_{\omega_3}, &
X_{\ee_3}\cdot v_{\omega_3-\ee_3}&= -2\, v_{\omega_3}, &
\\
X_{\ee_1}\cdot v_{\ee_3-\omega_3}&= -2\, v_{\omega_3-\ee_2}, &
X_{\ee_2}\cdot v_{\ee_3-\omega_3}&= 2\, v_{\omega_3-\ee_1}, &
X_{\ee_3}\cdot v_{\ee_2-\omega_3}&= 2\, v_{\omega_3-\ee_1}, &
\\
X_{\ee_1}\cdot v_{\ee_2-\omega_3}&= -2\, v_{\omega_3-\ee_3}, &
X_{\ee_2}\cdot v_{\ee_1-\omega_3}&= -2\, v_{\omega_3-\ee_3}, &
X_{\ee_3}\cdot v_{\ee_1-\omega_3}&= 2\, v_{\omega_3-\ee_2}, &
\\
X_{\ee_1}\cdot v_{-\omega_3}&= -2\, v_{\ee_1-\omega_3}, &
X_{\ee_2}\cdot v_{-\omega_3}&= 2\, v_{\ee_2-\omega_3}, &
X_{\ee_3}\cdot v_{-\omega_3}&= -2\, v_{\ee_3-\omega_3}, &
\\[2mm]
X_{-\ee_1}\cdot v_{\omega_3}&= -\tfrac12\, v_{\omega_3-\ee_1}, &
X_{-\ee_2}\cdot v_{\omega_3}&= -\tfrac12\, v_{\omega_3-\ee_2}, &
X_{-\ee_3}\cdot v_{\omega_3}&= -\tfrac12\, v_{\omega_3-\ee_3}, &
\\
X_{-\ee_1}\cdot v_{\omega_3-\ee_2}&= -\tfrac12\, v_{\ee_3-\omega_3}, &
X_{-\ee_2}\cdot v_{\omega_3-\ee_1}&= \tfrac12\, v_{\ee_3-\omega_3}, &
X_{-\ee_3}\cdot v_{\omega_3-\ee_1}&= \tfrac12\, v_{\ee_2-\omega_3}, &
\\
X_{-\ee_1}\cdot v_{\omega_3-\ee_3}&= -\tfrac12\, v_{\ee_2-\omega_3}, &
X_{-\ee_2}\cdot v_{\omega_3-\ee_3}&= -\tfrac12\, v_{\ee_1-\omega_3}, &
X_{-\ee_3}\cdot v_{\omega_3-\ee_2}&= \tfrac12\, v_{\ee_1-\omega_3}, &
\\
X_{-\ee_1}\cdot v_{\ee_1-\omega_3}&= -\tfrac12\, v_{-\omega_3}, &
X_{-\ee_2}\cdot v_{\ee_2-\omega_3}&= \tfrac12\, v_{-\omega_3}, &
X_{-\ee_3}\cdot v_{\ee_3-\omega_3}&= -\tfrac12\, v_{-\omega_3}, &
\\[2mm]
X_{\ee_1-\ee_2}\cdot v_{\omega_3-\ee_1}&= - v_{\omega_3-\ee_2}, &
X_{\ee_1-\ee_3}\cdot v_{\omega_3-\ee_1}&= - v_{\omega_3-\ee_3}, &
X_{\ee_2-\ee_3}\cdot v_{\omega_3-\ee_2}&= - v_{\omega_3-\ee_3}, &
\\
X_{\ee_1-\ee_2}\cdot v_{\ee_2-\omega_3}&= - v_{\ee_1-\omega_3}, &
X_{\ee_1-\ee_3}\cdot v_{\ee_3-\omega_3}&=  v_{\ee_1-\omega_3}, &
X_{\ee_2-\ee_3}\cdot v_{\ee_3-\omega_3}&= - v_{\ee_2-\omega_3}, &
\\[2mm]
X_{-\ee_1+\ee_2}\cdot v_{\omega_3-\ee_2}&= - v_{\omega_3-\ee_1}, &
X_{-\ee_1+\ee_3}\cdot v_{\omega_3-\ee_3}&= - v_{\omega_3-\ee_1}, &
X_{-\ee_2+\ee_3}\cdot v_{\omega_3-\ee_3}&= - v_{\omega_3-\ee_2}, &
\\
X_{-\ee_1+\ee_2}\cdot v_{\ee_1-\omega_3}&= - v_{\ee_2-\omega_3}, &
X_{-\ee_1+\ee_3}\cdot v_{\ee_1-\omega_3}&=  v_{\ee_3-\omega_3}, &
X_{-\ee_2+\ee_3}\cdot v_{\ee_2-\omega_3}&= - v_{\ee_3-\omega_3}, &
\\[2mm]
X_{\ee_1+\ee_2}\cdot v_{\ee_3-\omega_3}&= 4\, v_{\omega_3}, &
X_{\ee_1+\ee_3}\cdot v_{\ee_2-\omega_3}&= 4\, v_{\omega_3}, &
X_{\ee_2+\ee_3}\cdot v_{\ee_1-\omega_3}&= 4\, v_{\omega_3}, &
\\
X_{\ee_1+\ee_2}\cdot v_{-\omega_3}&= 4\, v_{\omega_3-\ee_3}, &
X_{\ee_1+\ee_3}\cdot v_{-\omega_3}&= -4\, v_{\omega_3-\ee_2}, &
X_{\ee_2+\ee_3}\cdot v_{-\omega_3}&= 4\, v_{\omega_3-\ee_1}, &
\\[2mm]
X_{-\ee_1-\ee_2}\cdot v_{\omega_3}&= \tfrac14\, v_{\ee_3-\omega_3}, &
X_{-\ee_1-\ee_3}\cdot v_{\omega_3}&= \tfrac14\, v_{\ee_2-\omega_3}, &
X_{-\ee_2-\ee_3}\cdot v_{\omega_3}&= \tfrac14\, v_{\ee_1-\omega_3}, &
\\
X_{-\ee_1-\ee_2}\cdot v_{\omega_3-\ee_3}&= \tfrac14\, v_{-\omega_3}, &
X_{-\ee_1-\ee_3}\cdot v_{\omega_3-\ee_2}&= -\tfrac14\, v_{-\omega_3}, &
X_{-\ee_2-\ee_3}\cdot v_{\omega_3-\ee_1}&= \tfrac14\, v_{-\omega_3}. &
\end{aligned}
$

\

\end{table}

\subsection{Some low Laplace eigenvalues of $\Grass{3}{8}$}
\label{subsec4:loweigenvalues}

The main goal of this section is to obtain the eigenvalues in $\Spec(\Grass{3}{8},g_{r})$ contributed via \eqref{eq:spec} by the irreducible representations $\pi_{\omega_3}$ and $\pi_{\omega_1+\omega_3}$ of $G=\Spin(7)$ with highest weights $\omega_3=\tfrac12(\ee_1+\ee_2+\ee_3)$ and $\omega_1+\omega_3=\tfrac12(3\ee_1+\ee_2+\ee_3)$ respectively. 
We can deal with both simultaneously because the decomposition $\pi_{\omega_1}\otimes \pi_{\omega_3} \simeq \pi_{\omega_1+\omega_3} \oplus \pi_{\omega_3}$, which ensures that there are $G$-invariant subspaces $V_{\pi_{\omega_1+\omega_3}}$ and $V_{\pi_{\omega_3}}$ of $V_{\pi_{\omega_1}}\otimes V_{\pi_{\omega_3}}$, which are irreducible as $G$-modules with highest weights $\omega_1+\omega_3$ and $\omega_3$ respectively, satisfying that
\begin{equation}\label{eq4:tensor-omega1-omega3}
V_{\pi_{\omega_1}}\otimes V_{\pi_{\omega_3}}
=V_{\pi_{\omega_1+\omega_3}} \oplus V_{\pi_{\omega_3}}
. 
\end{equation}

Of course, $\{u_{\mu}\otimes v_{\eta}: \mu\in\PP(\pi_{\omega_1}),\,  \eta\in\PP(\pi_{\omega_3})\}$ is a basis of $V_{\pi_{\omega_1}}\otimes V_{\pi_{\omega_3}}$. 
Remember the action of $X\in\fg_\C$ is given by $X\cdot u\otimes v=\big(\pi_{\omega_1}(X)\cdot u\big)\otimes v + u\otimes \big(\pi_{\omega_3}(X)\cdot v\big)$. 
Therefore, Tables~\ref{table:standard}--\ref{table:spin} allow us to compute $X_\alpha\cdot u_\mu\otimes v_\eta$ for every root $\alpha\in \Phi(\fg_\C,\ft_\C)$. 
Furthermore, it turns out that $u_{\mu}\otimes v_{\eta}$ is a weight vector of weight $\mu+\eta$, thus $H\cdot u_{\mu}\otimes v_{\eta}=(\mu+\eta)(H)\, u_{\mu}\otimes v_{\eta}$. 
Note that, unlike $\pi_{\omega_1}$ and $\pi_{\omega_3}$, there are weights of $\pi_{\omega_1}\otimes\pi_{\omega_3}$ with multiplicity greater than one.

\begin{lemma}\label{lem4:(Vomega1xVomega3)^H}
We have that 
$(V_{\pi_{\omega_1}}\otimes V_{\pi_{\omega_3}})^H=\Span_{\C}\{w_1, w_2\}$, where 
\begin{align*}
w_1&
=   u_{\ee_1}\otimes v_{-\omega_3}
- 4 u_{-\ee_2}\otimes v_{\omega_3-\ee_3}
+ 2 u_{\ee_2}\otimes v_{\ee_3-\omega_3}
+ 8 u_{-\ee_1}\otimes v_{\omega_3},
\\
w_2&
=  u_{\ee_3}\otimes v_{\ee_2-\omega_3}
- 2 u_{0}\otimes v_{\omega_3-\ee_1}
+ u_{0}\otimes v_{\ee_1-\omega_3}
+ 2 u_{-\ee_3}\otimes v_{\omega_3-\ee_2}
.
\end{align*}
\end{lemma}

\begin{proof}
We have that
\begin{multline}
(V_{\pi_{\omega_1}}\otimes V_{\pi_{\omega_3}})^{H}
=(V_{\pi_{\omega_1}}\otimes V_{\pi_{\omega_3}})^{\fh_\C}
\\ 
\subset 
(V_{\pi_{\omega_1}}\otimes V_{\pi_{\omega_3}})^{\fh_\C\cap\ft_\C}
=\bigoplus_{ \substack{
	\mu\in\mathcal P(\pi_{\omega_1}), \eta\in\mathcal P(\pi_{\omega_3})
	\\
	(\mu+\eta)(H)=0\;\forall\, H\in \fh_\C\cap\ft_\C
}}
(V_{\pi_{\omega_1}}\otimes V_{\pi_{\omega_3}})(\mu+\eta)
\\  
= (V_{\pi_{\omega_1}}\otimes V_{\pi_{\omega_3}}) \big(\tfrac12(\ee_1-\ee_2-\ee_3)\big) \oplus 
(V_{\pi_{\omega_1}}\otimes V_{\pi_{\omega_3}}) \big(\tfrac12(-\ee_1+\ee_2+\ee_3)\big)
\\ 
= 
\Span_{\C}\left\{ 
\begin{array}{l}
u_{\ee_1}\otimes v_{-\omega_3},\,
u_{-\ee_2}\otimes v_{\omega_3-\ee_3},\,
u_{\ee_2}\otimes v_{\ee_3-\omega_3},\,
u_{-\ee_1}\otimes v_{\omega_3},
\\ 
u_{\ee_3}\otimes v_{\ee_2-\omega_3},\,
u_{0}\otimes v_{\omega_3-\ee_1},\,
u_{0}\otimes v_{\ee_1-\omega_3},\,
4\, u_{-\ee_3}\otimes v_{\omega_3-\ee_2}
\end{array}
\right\}
.
\end{multline}
In the penultimate identity was used that $\{H_1-H_2+2H_3, H_1+H_2\}$ is a basis of $\fh_\C\cap\ft_\C$. 
The last identity follows by finding all ways to write $\pm\tfrac12(\ee_1-\ee_2-\ee_3)=\mu+\eta$ with $\mu\in\PP(\pi_{\omega_1})$ and $\eta\in\PP(\pi_{\omega_3})$. 

In order to determine $(V_{\pi_{\omega_1}}\otimes V_{\pi_{\omega_3}})^{H}$, it remains to find which $\C$-linear combinations of these 8 elements are vanished by the generators of $\fh_\C$ as in \eqref{eq4:h_C}, namely, $Y_{\pm\alpha_1}=X_{\pm(\ee_1-\ee_2)}+X_{\pm\ee_3}$ and $Y_{\pm(3\alpha_1+2\alpha_2)}=X_{\pm(\ee_1+\ee_2)}$. 
More precisely, we look for $a_{\pm\ee_i}$ for $i=1,2,3$ and $a_{0}^\pm$ in $\C$ such that the element
\begin{equation}
\begin{aligned}
w&:=
a_{\ee_1}\, u_{\ee_1}\otimes v_{-\omega_3}
+a_{-\ee_2}\, u_{-\ee_2}\otimes v_{\omega_3-\ee_3}
+a_{\ee_2}\, u_{\ee_2}\otimes v_{\ee_3-\omega_3}
+a_{-\ee_1}\, u_{-\ee_1}\otimes v_{\omega_3}
\\ &\quad
+a_{\ee_3}\, u_{\ee_3}\otimes v_{\ee_2-\omega_3}
+a_{0}^-\, u_{0}\otimes v_{\omega_3-\ee_1}
+a_{0}^+\, u_{0}\otimes v_{\ee_1-\omega_3}
+a_{-\ee_3}\,4\, u_{-\ee_3}\otimes v_{\omega_3-\ee_2}
\end{aligned}
\end{equation}
satisfies $X_{\pm(\ee_1+\ee_2)}\cdot w=0$ and  $(X_{\pm(\ee_1-\ee_2)}+X_{\pm\ee_3})\cdot w=0$. 
This long but straightforward procedure returns the conditions
$a_{-\ee_2}=-4a_{\ee_1}$,
$a_{\ee_2}=2 a_{\ee_1}$,
$a_{-\ee_1}= 8a_{\ee_1}$,
$a_{\ee_3} = a_{0}^+$,
$a_{-\ee_3}= 2a_{0}^+$, and  
$a_{0}^-= - 2a_{0}^+$, 
which completes the proof. 
\end{proof}

\begin{remark}
One can check via long calculations that $w_1+2w_2\in V_{\pi_{\omega_3}}^H$ and $-3w_1+8w_2\in V_{\pi_{\omega_1+\omega_3}}^H$, though it will not be necessary. 
\end{remark}

We are now ready to obtain explicit expressions for Casimir scalars and the tricky term. 
From \eqref{eq4:lambda^pi_Lambda}, it follows that 
\begin{align}
\lambda_{\kil}^{\pi_{\omega_3}}&=\frac{21}{40},&
\lambda_{\kil}^{\pi_{\omega_1+\omega_3}}&=\frac{49}{40}. 
\end{align}

The branching law from $G=\Spin(7)$ to $K=\operatorname{G}_2$ gives 
\begin{align}\label{eq4:branching-Spin(7)toG2}
\pi_{\omega_3}|_{K} &\simeq \tau_{0}\oplus\tau_{\nu_1},&
\pi_{\omega_1+\omega_3}|_{K} &\simeq \tau_{\nu_1} \oplus \tau_{\nu_2} \oplus \tau_{2\nu_1}
.
\end{align}
Similarly, the branching law from $K$ to $H$ of the irreducible components appeared above give
\begin{equation}\label{eq4:branching-G2toSO(4)}
\begin{aligned}
\tau_{0}|_H& 
=\sigma_0\otimes\sigma_0,\\
\tau_{\nu_1}|_H& 
=\sigma_1\otimes\sigma_1\oplus \sigma_2\otimes\sigma_0, \\
\tau_{\nu_2}|_H &
=\sigma_2\otimes\sigma_0\oplus \sigma_0\otimes\sigma_2 \oplus \sigma_3\otimes\sigma_1, \\ 
\tau_{2\nu_1}|_H &
= \sigma_0\otimes\sigma_0\oplus \sigma_1\otimes\sigma_1 \oplus \sigma_2\otimes\sigma_2 \oplus \sigma_3\otimes\sigma_1 \oplus \sigma_4\otimes\sigma_0.  
\end{aligned}
\end{equation}
\cite{Sage} calculates them as follows: 
\smallskip
\begin{lstlisting}
sage: G=WeylCharacterRing("B3", style="coroots")
sage: K=WeylCharacterRing("G2", style="coroots")
sage: H=WeylCharacterRing("A1xA1", style="coroots")
sage: b1=branching_rule(G,K,"miscellaneous")
sage: b2=branching_rule(K,H,"extended")
sage: omega=G.fundamental_weights()
sage: nu=K.fundamental_weights()
sage: 
sage: print("branching to K of G(omega3):")
sage: print(G(omega[3]).branch(K,rule=b1))
sage: print("branching to K of G(omega1+omega3):")
sage: print(G(omega[1]+omega[3]).branch(K,rule=b1))
sage: print("----")
sage: print("branching to H of K(0):")
sage: print(K(0*nu[1]).branch(H,rule=b2))
sage: print("branching to H of K(nu1):")
sage: print(K(nu[1]).branch(H,rule=b2))
sage: print("branching to H of K(nu2):")
sage: print(K(nu[2]).branch(H,rule=b2))
sage: print("branching to H of K(2nu1):")
sage: print(K(2*nu[1]).branch(H,rule=b2))
branching to K of G(omega3):
G2(0,0) + G2(1,0)
branching to K of G(omega1+omega3):
G2(1,0) + G2(0,1) + G2(2,0)
----
branching to H of K(0):
A1xA1(0,0)
branching to H of K(nu1):
A1xA1(1,1) + A1xA1(2,0)
branching to H of K(nu2):
A1xA1(2,0) + A1xA1(3,1) + A1xA1(0,2)
branching to H of K(2nu1):
A1xA1(0,0) + A1xA1(1,1) + A1xA1(2,2) + A1xA1(3,1) + A1xA1(4,0)
\end{lstlisting}
\smallskip
Here, for non-negative integers $a,b$, \texttt{G2(a,b)} and \texttt{A1xA1(a,b)} means in our notation $\tau_{a\nu_1+b\nu_2}$ and $\sigma_a\otimes \sigma_{b}$ respectively.

For any $\tau\in \widehat K$, $\dim V_\tau^H$ is the number of times that the trivial representation $\sigma_0\otimes\sigma_0$ of $H$ appears in $\tau|_H$. 
It follows immediately from \eqref{eq4:branching-Spin(7)toG2} and \eqref{eq4:branching-G2toSO(4)} that 
\begin{equation}\label{eq4:dim-pi_omega3^H=dim-pi_omega1+omega3^H=1}
\begin{aligned}
\dim V_{\pi_{\omega_3}}^H&
	= \dim V_{\tau_0}^H+\dim V_{\tau_{\nu_1}}^H=1
,\\ 
\dim V_{\pi_{\omega_1+\omega_3}}^H&
	= \dim V_{\tau_{\nu_1}}^H +\dim V_{\tau_{\nu_2}}^H + \dim V_{\tau_{2\nu_1}}^H
=1
.
\end{aligned}
\end{equation}
Consequently, we are in the situation of Remark~\ref{rem:Upsilon-dimV_pi^H=1} for $\pi_{\omega_3}$ and $\pi_{\omega_1+\omega_3}$. 
Moreover, using \eqref{eq4:lambda^tau_mu}, we have that $\pi_{\omega_3}\big(-\Cas_{\fk,\kil|_{\fk}}\big)$ acts on $V_{\pi_{\omega_3}}^H$ and $V_{\pi_{\omega_1+\omega_3}}^H$ by multiplication by the scalars 
\begin{align}
\lambda_{\kil|_\fk}^{\tau_0}=0\text{ and }
\lambda_{\kil|_\fk}^{\tau_{2\nu_1}}=\frac{14}{15}
\text{ respectively.}
\end{align}

Long and tedious calculations using \eqref{eq4:trickyterm} give 
\begin{equation}
\Upsilon_{\pi_{\omega_1}\otimes\pi_{\omega_3}}(w_i)=\frac{9}{40}\,w_i\quad\text{for }i=1,2,
\end{equation} 
where $w_1$ and $w_2$ are as in Lemma~\ref{lem4:(Vomega1xVomega3)^H}. 
Equivalently,  $\Upsilon_{\pi_{\omega_1}\otimes\pi_{\omega_3}}|_{(V_{\pi_{\omega_1}}\otimes V_{\pi_{\omega_3}})^H}=\frac{9}{40}\, \Id_{(V_{\pi_{\omega_1}}\otimes V_{\pi_{\omega_3}})^H}$. 
Moreover, \eqref{eq4:tensor-omega1-omega3} forces $\Upsilon_{\pi_{\omega_1}}$ and $\Upsilon_{\pi_{\omega_3}}$ acts by the scalar $\frac{9}{40}$ on $V_{\pi_{\omega_1}}^H$ and $V_{\pi_{\omega_3}}^H$ respectively. 
In other words, $\upsilon^{\pi_{\omega_3}} = \upsilon^{\pi_{\omega_1+\omega_3}}= \frac{9}{40}$ in the notation of Remark~\ref{rem:Upsilon-dimV_pi^H=1}.

We are now ready to obtain explicit expressions for the eigenvalues $\lambda_1^{\pi_{\omega_3}}(r)$ and $\lambda_1^{\pi_{\omega_1+\omega_3}}(r)$ via the formula
$\lambda_1^{\pi}(r) 
= \big(\lambda_{\kil}^{\pi} -\upsilon^{\pi} -\lambda_{\kil|_\fk}^{\tau} \big)\, r_1^2
+ \upsilon^{\pi} \, r_2^2
+ \lambda_{\kil|_\fk}^{\tau} \, r_3^2
$ 
in \eqref{eq:threeterms} for any $r=(r_1,r_2,r_3)\in\R_{>0}^3$:
\begin{equation}\label{eq4:threeterms}
\begin{aligned}
\lambda_1^{\pi_{\omega_3}}(r) &
= \left(\frac{21}{40} -\frac{9}{40}  \right)\, r_1^2
+ \frac{9}{40} \, r_2^2
= \frac{3}{10} \, r_1^2
+ \frac{9}{40} \, r_2^2
,
\\
\lambda_1^{\pi_{\omega_1+\omega_3}}(r) &
= \left(\frac{49}{40} -\frac{9}{40} -\frac{14}{15} \right)\, r_1^2
+ \frac{9}{40} \, r_2^2
+ \frac{14}{15} \, r_3^2
= \frac{1}{15}\, r_1^2
+ \frac{9}{40} \, r_2^2
+ \frac{14}{15} \, r_3^2
.
\end{aligned}
\end{equation}

\subsection{Spectral uniqueness for $\Grass{3}{8}$}

We are now ready to prove Theorem~\ref{thm0:main} for $\Grass{3}{8}$, namely, every symmetric metric on $\Grass{3}{8}\simeq G/H$, with $G=\Spin(7)$ and $H=\SO(4)$, is spectrally unique within the space of $G$-invariant metrics on $\Grass{3}{8}$.

According to \cite[\S6]{Kerr96}, the symmetric metrics on $G/H\simeq \Grass{3}{8}$ are 
\begin{equation}
\left\{\bar g_t:= g_{(\sqrt{12} t,\sqrt{4} t,\sqrt{3} t)}:t>0 \right\}.
\end{equation}

\begin{remark}\label{rem4:notationMegan}
In the notation in \cite{Kerr96}, $\fp_1$ and $\fp_2$ are interchanged and $Q=10\, \kil_\fg$, so $x_1=\frac{10}{r_2^2}$, $x_2=\frac{10}{r_1^2}$, and $x_3=\frac{10}{r_3^2}$.
\end{remark}

The standard symmetric space $\big(\frac{\SO(8)}{\SO(5)\times\SO(3)}, g_{\kil_{\so(8)}}\big)$ is isometric to $\bar g_t$ for $t=\sqrt{5/18}$.
One has $\lambda_1\big(\frac{\SO(8)}{\SO(5)\times\SO(3)}, g_{\kil_{\so(8)}}\big)=\frac{5}{4}$; it is attained at the representation $\bigwedge^3\C^8$ of $\SO(8)$, so its multiplicity is $\binom{8}{3}=56$ (see \cite[Table A.2]{Urakawa86}).
Note that 
$\lambda^{\pi_{\omega_1}}(\sqrt{10/3}, \sqrt{10/9}, \sqrt{5/6})
=\lambda^{\pi_{\omega_1+\omega_3}}(\sqrt{10/3}, \sqrt{10/9}, \sqrt{5/6})=\frac{5}{4}$ 
and $\dim \pi_{\omega_1}+\dim\pi_{\omega_1+\omega_3}=7+49=56$, which implies that any eigenvalue of the Laplace-Beltrami operator of $(\Grass{3}{8},\bar g_{t})$ coming from $\pi\in\widehat G\smallsetminus\{1_G,\pi_{\omega_1},\pi_{\omega_1+\omega_3}\}$ is strictly greater than $\lambda_1(\Grass{3}{8},\bar g_{t})$, for any $t>0$.

\begin{theorem}
Any $G$-invariant metric on $\Grass{3}{8}\simeq G/H$ isospectral to a symmetric metric on $\Grass{3}{8}$ is in fact isometric to such symmetric metric. 
\end{theorem}

\begin{proof}
Suppose that $\Spec(\Grass{3}{8},g_{r})=\Spec(\Grass{3}{8},\bar g_t)$ for some $r=(r_1,r_2,r_3)\in\R_{>0}^3$ and $t>0$. 
Without loosing generality, we can assume that $t=1$, that is, $\bar g_1=g_{r_0}$ with $r_0=(\sqrt{12},2,\sqrt{3})$. 
The goal is to show that $r=r_0$.

The multiplicity of 
$
\bar \lambda_1:=
\lambda_1(\Grass{3}{8},\bar g_{1})
$ in $\Spec(\Grass{3}{8},\bar g_{1})$ is $56$. 
Therefore $\bar \lambda_1$ is in $\Spec(\Grass{3}{8},g_{r})$ with multiplicity $56$, thus \eqref{eq:spec} forces 
\begin{equation}\label{eq4:multiplicidad}
56=\sum_{\pi\in\widehat G_H:\,  \bar\lambda_1\in\Spec(\pi(-C_{g_r}))|_{V_\pi^H} }  \dim V_\pi\, a_\pi,
\end{equation}
where $a_\pi$ denotes the multiplicity of $\bar\lambda_1$ in $\Spec(\pi(-C_{g_r}))|_{V_\pi^H}$, so $0\leq a_\pi\leq \dim V_\pi^H$. 
One can easily check that the only irreducible representations of $G$ of dimension at most $56$ with $\dim_{\pi}^H>0$ are $1_G$, $\pi_{2\omega_1}$, $\pi_{\omega_3}$, $\pi_{\omega_1+\omega_3}$, and $\pi_{2\omega_3}$. 
We know that the eigenvalue associated to the trivial representation $1_G$ is $0$. 
Furthermore, $\dim V_{\pi_{\omega_3}}^H=\dim V_{\pi_{\omega_1+\omega_3}}^H=1$ by \eqref{eq4:dim-pi_omega3^H=dim-pi_omega1+omega3^H=1} and one can check that $\dim V_{\pi_{2\omega_1}}^H=1$ and $\dim V_{\pi_{2\omega_3}}^H=2$. 
Now, equation \eqref{eq4:multiplicidad} becomes 
\begin{equation}
\begin{aligned}
56&
=\dim V_{\pi_{2\omega_1}}\, a_{\pi_{2\omega_1}}
+\dim V_{\pi_{\omega_3}}\, a_{\pi_{\omega_3}}
+\dim V_{\pi_{\omega_1+\omega_3}}\, a_{\pi_{\omega_1+\omega_3}}
+\dim V_{\pi_{2\omega_3}}\, a_{\pi_{2\omega_3}}
\\ &
=27 a_{\pi_{2\omega_1}}
+8 a_{\pi_{\omega_3}}
+48 a_{\pi_{\omega_1+\omega_3}}
+35 a_{\pi_{2\omega_3}}
\\ & 
\text{
	with $0\leq a_{\pi_{2\omega_1}},a_{\pi_{\omega_3}}, a_{\pi_{\omega_1+\omega_3}}\leq 1$ and $0\leq a_{\pi_{2\omega_3}}\leq 2$,
}
\end{aligned} 
\end{equation}
which clearly implies that $a_{\pi_{\omega_3}}= a_{\pi_{\omega_1+\omega_3}}=1$ and $a_{\pi_{2\omega_1}}=a_{\pi_{2\omega_3}}=0$. 
Hence 
\begin{equation}
\bar \lambda_1
=\lambda_1^{\pi}(\sqrt{12},2,\sqrt{3})
=\lambda_1^{\pi}(r)
\qquad\text{ for }\pi\in\{\pi_{\omega_3}, \pi_{\omega_1+\omega_3} \}
.
\end{equation}
By \eqref{eq4:threeterms}, we have that
\begin{equation}\label{eq4:igualdad-autovalores}
\begin{aligned}
\begin{cases}
\tfrac92 = \frac{3}{10} \, r_1^2 + \frac{9}{40} \, r_2^2, 
\\
\tfrac92 =	\frac{1}{15}\, r_1^2 + \frac{9}{40} \, r_2^2 + \frac{14}{15}\, r_3^2, 
\end{cases}
\iff \qquad &
\begin{cases}
\tfrac92 = \frac{3}{10} \, r_1^2 + \frac{9}{40} \, r_2^2, 
\\
r_1^2=4\, r_3^2.
\end{cases}
\end{aligned}
\end{equation}

We now analyze the volume, which is also determined by the spectra. 
We have that 
\begin{equation}
\vol(\Grass{3}{8},g_r)
=r_1^{\dim\fp_1} r_2^{\dim\fp_2} r_3^{\dim\fp_3} \vol(\Grass{3}{8},g_{(1,1,1)})
=r_1^{4} r_2^{3} r_3^{8} \vol(\Grass{3}{8},g_{(1,1,1)}).
\end{equation}
Now, since $\vol(\Grass{3}{8},g_r)=\vol(\Grass{3}{8},g_{r_0})$, we obtain $r_1^{4} r_2^{3} r_3^{8}=12^2\cdot 8\cdot 3^4=2^7\cdot 3^6$. 
Substituting $r_1^2=4r_3^2$ in it, we can assert that
\begin{equation}\label{eq4:igualdadpto}
r_2^{3}\,  r_3^{12} =2^{3}3^6 
\qquad\Longrightarrow\qquad
r_2\,  r_3^{4} =18
.
\end{equation}
It follows form \eqref{eq4:igualdad-autovalores} that 
\begin{equation}\label{eq4:r_3^2}
r_3^2
=\frac{1}{4}r_1^2
=\frac{1}{4}\frac{10}{3} \left(\frac92-\frac{9}{40}r_2^2\right)
=\frac{3}{16}(20-r_2^2).
\end{equation}
By replacing this expression for $r_3^2$ in \eqref{eq4:igualdadpto}, we get 
$0= r_2\frac{9}{2^8}(20-r_2^2)^2-18 $, so 
$0= r_2(20-r_2^2)^2-2^9 $.

It is a simple calculus exercise to show that the polynomial $f(x):=x(20 - x)^2-2^9 $ has two positive roots: $x_1=2$ and $x_2\approxeq 5.44915345$.
The second one gives $r_2\approxeq  5.44915345$, and \eqref{eq4:r_3^2} implies $r_3^2\approxeq -1.817488<0$, which is not possible. 
We conclude that $r_2=2$, thus $r_3^2=3$ by \eqref{eq4:r_3^2} and $r_1^2=12$ by \eqref{eq4:igualdad-autovalores}, and the proof is complete. 
\end{proof}

\subsection{Homogeneous Einstein metrics on $\Grass{3}{8}$}
The Einstein symmetric metric $g_{\kil_{\so(8)}}$ on $\frac{\SO(8)}{\SO(5)\times\SO(3)}$ is $\nu$-stable according to Cao and He~\cite{CaoHe15}. 
Kerr proved in \cite{Kerr96} that there are (up to scaling) two additional $G$-invariant Einstein metrics $g_1,g_2$ on $\Grass{3}{8}$.
Their approximate parameters $x=(x_1,x_2,x_3)$ and $r=(r_1,r_2,r_3)$ (see Remark~\ref{rem4:notationMegan} for their relation) are as follows: 
\begin{equation}
\renewcommand{\arraystretch}{1.2}
\begin{array}{ccc}
& g_1 & g_2
\\ \hline 
x_1&0.902191989660862& 0.369813422882157
\\
x_2&0.425178535419486& 1.10029990844058
\\
x_3&1&1 
\\
r_1^2&23.5195316013163& 9.08843118434199
\\
r_2^2&11.0841152599449& 27.0406626186377
\\
r_3^2&10&10
\end{array}
\end{equation}

By using the formula for the scalar curvature in \cite[page 168]{Kerr96}, we obtain that 
\begin{equation}
\begin{aligned}
\scal(\Grass{3}{8},g_1)&\approxeq 75.1030942567225,& \scal(\Grass{3}{8},g_2)&\approxeq 68.5963932678592.
\end{aligned}
\end{equation}
Dividing by $\dim\Grass{3}{8}=15$, we obtain the Einstein constants $E_1\approxeq 5.00687295044817$ and $E_2\approxeq 4.57309288452394$ of $g_1$ and $g_2$ respectively.  

Although we do not have an explicit expression for the first Laplace eigenvalue of $g_1$ and $g_2$, we have that 
\begin{equation}
\begin{aligned}
\lambda_1(\Grass{3}{8},g_1)&\leq \lambda^{\pi_{\omega_3}}(r(g_1))\approxeq 9.54978541388250< 10.0137459008963 \approxeq 2E_1,
\\
\lambda_1(\Grass{3}{8},g_2)&\leq \lambda^{\pi_{\omega_3}}(r(g_2))\approxeq 8.81067844449609 < 9.14618576904789 \approxeq 2E_2.
\end{aligned}
\end{equation}
We conclude that the Einstein manifolds $(\Grass{3}{8},g_1)$ and $(\Grass{3}{8},g_2)$ are $\nu$-unstable, and therefore dynamically unstable.

\begin{remark}
With a similar strategy as in Subsection~\ref{subsec4:loweigenvalues}, the authors obtained that
\begin{align*}
\lambda_1^{\pi_{2\omega_1}}(r)&
= \frac{7}{15}r_1^2+\frac{14}{15}r_3^2
,
\\
\lambda_1^{\pi_{2\omega_3}}(r)&
=\frac7{30} r_1^2 + \frac9{30} r_2^2 + \frac{7}{15} r_3^2 
- \frac1{30} \sqrt{121r_1^4  + 81r_2^4 + 196r_3^4 - 90r_1^2r_2^2 - 92r_1^2r_3^2 - 180r_2^2r_3^2 }
,
\\
\lambda_2^{\pi_{2\omega_3}}(r)&
=\frac7{30} r_1^2 + \frac9{30} r_2^2 + \frac{7}{15} r_3^2 
+ \frac1{30} \sqrt{121r_1^4  + 81r_2^4 + 196r_3^4 - 90r_1^2r_2^2 - 92r_1^2r_3^2 - 180r_2^2r_3^2 }
.
\end{align*}
The case $\pi_{2\omega_3}$ was particularly hard because $\dim V_{\pi_{2\omega_3}}^H=2$. 
Moreover, the eigenbasis of $\pi_{2\omega_3}(-C_{g_r})|_{V_{\pi_{2\omega_3}}^H}$ depends on $r$. 
That is, there are eigenvectors in $V_{\pi_{2\omega_3}}^H$ of the Casimir operators $\pi_{2\omega_3}\big(-\Cas_{\fg,\kil}\big)$ and $\pi_{2\omega_3}\big(-\Cas_{\fk,\kil|_{\fk}}\big)$, that are not eigenvectors of $\Upsilon_{\pi_{2\omega_3}}$

The smallest positive eigenvalue $\lambda_1(\Grass{3}{8},g_{r})$ of the Laplace-Beltrami operator associated to $(\Grass{3}{8},g_{r})$ might be equal to $\min\{\lambda_1^{\pi_{\omega_3}}(r), \lambda_1^{\pi_{\omega_1+\omega_3}}(r) ,  \lambda_1^{\pi_{2\omega_3}}(r)\}$. 
To establish it, one has to show that $\lambda_i^\pi(r)\geq \min\{\lambda_1^{\pi_{\omega_3}}(r), \lambda_1^{\pi_{\omega_1+\omega_3}}(r) ,  \lambda_1^{\pi_{2\omega_3}}(r)\}$ for every $1\leq i\leq \dim V_\pi^H$ and $\pi\in\widehat G_H$. 
The difficult cases are those $\pi\in\widehat G$ satisfying $\dim V_\pi^H>1$ (e.g.\ $\pi_{2\omega_3}$), since it is not easy to determine $V_{\pi}^H$ and either the eigenbasis for $\pi(-C_{g_r})|_{V_\pi^H}$ due to the same reason explained in the previous paragraph. 
\end{remark}

\bibliographystyle{plain}

\begin{thebibliography}{BLP22}

\bibitem[Ar93]{Arvanitoyeorgos93}
	{\sc A. {Arvanitoyeorgos}}.
	{\it New invariant {E}instein metrics on generalized flag manifolds.}
	Trans. Amer. Math. Soc. \textbf{337}:2 (1993), 981--995.
	DOI: \href{http://dx.doi.org/10.2307/2154253} {10.2307/2154253}.

\bibitem[Be]{Besse}
	{\sc A. Besse}.
	{Einstein manifolds.}
	Class. Math. 
	Springer, Berlin, 2008.

\bibitem[BLP22]{BLPhomospheres}
	{\sc R. Bettiol, E.A. Lauret, P. Piccione}.
	{\it The first eigenvalue of a homogeneous CROSS.}
	J. Geom. Anal. \textbf{32} (2022), 76.
	DOI: \href{https://doi.org/10.1007/s12220-021-00826-7} {10.1007/s12220-021-00826-7}.

\bibitem[CH15]{CaoHe15}
	{\sc H.-D. Cao, C. He}.
	{\it Linear stability of Perelman's $\nu$-entropy on symmetric spaces of compact type.}
	J. Reine Angew. Math. \textbf{709} (2015), 229--246.
	DOI: \href{http://dx.doi.org/10.1515/crelle-2013-0096} {10.1515/crelle-2013-0096}.	

\bibitem[DZ79]{DAtriZiller}
	{\sc J.E. D'Atri, W. Ziller}.
	{\it Naturally reductive metrics and Einstein metrics on compact Lie groups.}
	Mem. Amer. Math. Soc. \textbf{18}:215, 1979.
	DOI: \href{https://doi.org/10.1090/memo/0215} {10.1090/memo/0215}.	

\bibitem[GSS10]{GordonSchuethSutton10}
	{\sc C. Gordon, D. Schueth, C. Sutton}.
	{\it Spectral isolation of bi-invariant metrics on compact Lie groups}.
	Ann. Inst. Fourier (Grenoble) \textbf{60}:5 (2010), 1617--1628.
	DOI: \href{http://dx.doi.org/10.5802/aif.2567} {10.5802/aif.2567}.

\bibitem[Ke96]{Kerr96}
    {\sc M.M. Kerr}.
    {\it Some new homogeneous {E}instein metrics on symmetric spaces.}
    Trans. Amer. Math. Soc. \textbf{348}:1 (1996), 153--171.
    DOI: \href{http://dx.doi.org/10.1090/S0002-9947-96-01512-7} {10.1090/S0002-9947-96-01512-7}.

\bibitem[Ki90]{Kimura}
    {\sc M. Kimura}.
    {\it Homogeneous {E}instein metrics on certain {K}\"{a}hler {$C$}-spaces.}
    In \textit{Recent topics in differential and analytic geometry}, 303--320, Adv. Stud. Pure Math. \textbf{18}-{\rm I}, 1990.
    DOI: \href{http://dx.doi.org/10.2969/aspm/01810303} {10.2969/aspm/01810303}.
    
\bibitem[Kn]{Knapp-book-beyond}
	{\sc A.W. Knapp}.
	{Lie groups beyond an introduction.}
	{\it Progr. Math.} \textbf{140}.
	Birkh\"auser Boston Inc., 2002.

\bibitem[Kr15]{Kroencke15}
	{\sc K. Kr\"{o}ncke}.
	{\it Stability and instability of {R}icci solitons.}
	Calc. Var. Partial Differential Equations \textbf{53}:1--2 (2015), 265--287.
	DOI: \href{https://doi.org/10.1007/s00526-014-0748-3} {10.1007/s00526-014-0748-3}.

\bibitem[La19a]{Lauret-SpecSU(2)}
	{\sc E.A. Lauret}.
	{\it The smallest Laplace eigenvalue of homogeneous 3-spheres}.
	Bull. Lond. Math. Soc. \textbf{51}:1 (2019), 49--69.
	DOI: \href{http://dx.doi.org/10.1112/blms.12213} {10.1112/blms.12213}.

\bibitem[La19b]{Lauret-globalrigid}
	{\sc E.A. Lauret}.
	{\it Spectral uniqueness of bi-invariant metrics on symplectic groups}.
	Transform. Groups \textbf{24}:4 (2019), 1157--1164.
	DOI: \href{http://dx.doi.org/10.1007/s00031-018-9486-5} {10.1007/s00031-018-9486-5}.
	
\bibitem[LR23]{LauretRodriguez-hearingsymm1}
	{\sc E.A. Lauret, J.S. Rodríguez}.
	{\it Spectrally distinguishing symmetric spaces I}.
	Accepted in Math. Z. in March 2025. 
	DOI: \href{https://doi.org/10.48550/arXiv.2311.09719} {10.48550/arXiv.2311.09719}.

\bibitem[LS88]{LevasseurSmith}
	{\sc T. Levasseur, S.P. Smith}.
	{\it Primitive ideals and nilpotent orbits in type {$G_2$}.}
	J. Algebra \textbf{114} (1988), 81--105.
	DOI: \href{http://dx.doi.org/10.1016/0021-8693(88)90214-1} {10.1016/0021-8693(88)90214-1}.

\bibitem[LSS21]{LinSchmidtSuttonII}
	{\sc S. Lin, B. Schmidt, and C.J. Sutton}.
	{\it Geometric structures and the {L}aplace spectrum, {P}art II}.
	Trans. Amer. Math. Soc. \textbf{374}:12 (2021), 8483--8530.
	DOI: \href{https://doi.org/10.1090/tran/8417} {10.1090/tran/8417}.

\bibitem[Ma97]{Mashimo97b}
	{\sc K. Mashimo}.
	{\it On branching theorem of the pair {{\((G_2,\text{SU}(3))\)}}.}
	Nihonkai Math. J. \textbf{8}:2 (1997), 101--107.

\bibitem[On66]{Onishchik66-inclusion}
	{\sc A.L. Onishchik}.
	{\it Inclusion relations among transitive compact transformation groups.}
	Transl., Ser. 2, Am. Math. Soc. \textbf{50} (1966), 5--58. 
	DOI: \href{http://dx.doi.org/10.1090/trans2/050/02} {10.1090/trans2/050/02}.

\bibitem[Sage]{Sage}
    {\sc The Sage Developers}.
    {\it {S}ageMath, the {S}age {M}athematics {S}oftware {S}ystem ({V}ersion 9.7)}.
    \url{https://www.sagemath.org}, 2022. 
    DOI: \href{https://doi.org/10.5281/zenodo.593563} {10.5281/zenodo.593563}.

\bibitem[SS14]{SchmidtSutton13}
	{\sc B. Schmidt, C. Sutton}
	{\it Detecting the moments of inertia of a molecule via its rotational spectrum, II}.
	Preprint available at \href{http://users.math.msu.edu/users/schmidt/} {Schmidt's web page.} (2014).

	
\bibitem[Ur86]{Urakawa86}
	{\sc H. Urakawa}.
	{\it The first eigenvalue of the {L}aplacian for a positively curved homogeneous {R}iemannian manifold.}
	Compositio Math. \textbf{59}:1 (1986), 57--71.

\end{thebibliography}

\end{document}